\documentclass[a4paper,11pt,reqno]{article}

\usepackage{tikz-cd} 

\usepackage{relsize}
\usepackage{cite}
\usepackage{color}
\usepackage{amsmath,amsthm,amssymb}
\usepackage{enumitem}
\usepackage[margin=2.6cm]{geometry}

\makeatletter
\usepackage{comment}
\let\wfs@comment@comment\comment
\let\comment\@undefined

\usepackage{changes}
\let\wfs@changes@comment\comment
\let\comment\@undefined

\newcommand\comment{%
    \ifthenelse{\equal{\@currenvir}{comment}}
    {\wfs@comment@comment}
    {\wfs@changes@comment}%
}

\newcommand\myup{\mathrel{\stackrel{\makebox[0pt]{\mbox{\normalfont\tiny up}}}{\sim}}}
\newcommand\mydn{\mathrel{\stackrel{\makebox[0pt]{\mbox{\normalfont\tiny dn}}}{\sim}}}
\makeatletter

\makeatletter
\def\namedlabel#1#2{\begingroup
    #2%
    \def\@currentlabel{#2}%
    \phantomsection\label{#1}\endgroup
}
\makeatother
\usepackage{scalerel,stackengine}
\stackMath
\usepackage{amsmath,amssymb,amsthm,mathrsfs,amsfonts,dsfont,stmaryrd} 

\usepackage{multirow}
\usepackage{amscd}   
\usepackage{fullpage}
\usepackage[all]{xy}  
\usepackage{mathrsfs,enumitem}
\usepackage[titletoc,toc,title]{appendix}
\usepackage{xcolor}  
\usepackage[normalem]{ulem}
\usepackage{diagbox}
\usepackage{caption}

\usepackage[colorlinks = true,
            linkcolor = black,
            urlcolor  = blue,
            citecolor = red,
            anchorcolor = blue]{hyperref}
\usepackage{graphicx}
 \usepackage{color}
 \usepackage{tikz, framed}
\usetikzlibrary{arrows,shapes,matrix,decorations.pathmorphing}
\usetikzlibrary{backgrounds}
\usepackage[english]{babel} 
\newtheorem{theorem}{Theorem} 

\newtheorem{lemma}[theorem]{Lemma}
\newtheorem{proposition}[theorem]{Proposition}
\newtheorem{corollary}[theorem]{Corollary}
\theoremstyle{definition}
\newtheorem{definition}[theorem]{Definition}
\newtheorem{remark}[theorem]{Remark}
\newtheorem{example}[theorem]{Example}
\newtheorem{notation}[theorem]{Notation}
\newcommand{\mL}{\mathcal{L}}

\newcommand{\mC}{\mathcal{C}}
\newcommand{\mP}{\mathcal{P}}

 \newcommand{\mB}{\mathcal{B}}
  
   \newcommand{\M}{\mathcal{M}}
 \newcommand{\Fq}{\mathbb{F}_q}

 \newcommand{\Z}{\mathbb{Z}}
\newcommand{\F}{\mathbb{F}}

\newcommand{\one}{{\mathbf{1}}}
\newcommand{\zero}{{\mathbf{0}}}

\newcommand\qbin[3]{\begin{small}\left[\begin{matrix} #1 \\ #2 \end{matrix} \right]_{#3}\end{small}} 
\DeclareMathOperator{\cl}{cl}

\DeclareMathOperator{\w}{w}
\DeclareMathOperator{\tr}{Tr}
\DeclareMathOperator{\h}{h}
\providecommand{\keywords}[1]
{
  	
  \noindent\textbf{{\small \textbf{Keywords.}}} #1
}

\title{Invariants of Tutte Partitions and a $q$-Analogue}

\usepackage{authblk}
\author[1]{Eimear Byrne}
\affil[1,2]{School of Mathematics and Statistics, University College Dublin, Belfield, Ireland}
\author[2]{Andrew Fulcher}

\date{}

\begin{document}

\maketitle

\begin{abstract}
   We describe a construction of the Tutte polynomial for both matroids and $q$-matroids based on an appropriate partition of the underlying support lattice into intervals that correspond to prime-free minors, which we call a Tutte partition. We show that such partitions in the matroid case include the class of partitions arising in Crapo's definition of the Tutte polynomial, while not representing a direct $q$-analogue of such partitions.
   We propose axioms of a $q$-Tutte-Grothendieck invariant and show that this yields a $q$-analogue of a Tutte-Grothendieck invariant. We establish the connection between the rank generating polynomial and the Tutte polynomial, showing that one can be obtained from the other by convolution.
\end{abstract}



\tableofcontents

%%%%%%%%%%%%%%%%%%%%%%

%%%%%%%%%%%%%%%%%%%%%%%%%%%%%%%%%%%%%

\section{Introduction}

The Tutte polynomial of a matroid is a fundamental invariant from which several other matroid invariants may be obtained. It is very well-studied and there are numerous papers on the topic; see, for example, \cite{britztutte,brylawski_oxley_1992,brylowski_ams,BERNARDI2022108355,CAMERON2022105584}. In particular, those functions of a matroid $M$ that can be expressed
recursively as functions on a pair of disjoint minors of $M$ found by deletion and contraction of $M$, respectively, may be expressed in terms of the Tutte polynomial of $M$. The Tutte polynomial of $M$ carries information common to all matroids in the same isomorphism class of $M$. For example, the number of bases, number of independent sets, and number of spanning sets of a matroid can be found by taking different evaluations of its Tutte polynomial. The characteristic polynomial of $M$ is obtained by an evaluation in one of the variables of its Tutte polynomial, while its rank function is obtained by a linear substitution of variables. The Tutte polynomial is an invariant of matroid duality: that is the Tutte polynomial of the dual matroid $M^*$ determined by the Tutte polynomial of $M$. In the case of representable matroids, the celebrated MacWilliams duality theorem in coding theory can be retrieved as a special case of a duality result on the Tutte polynomial.

While the origins of the Tutte polynomial lie in graph theory, Crapo defined this invariant for matroids by extending the notions of internal and external activity of the bases of a graph, defined originally by Tutte \cite{C69,tutte}, to the setting of matroids. One way of looking at Crapo's Tutte polynomial is to see it as a sum of monomial terms, each associated with a different element of a partition of the underlying Boolean lattice into intervals that contains a unique basis of the matroid. That different partitions of the lattice yield the same polynomial can be seen by the connection of this polynomial to the rank generating polynomial of the matroid via an invertible linear transformation, as shown by Crapo in \cite{C69}.

In \cite{brylowski_alguniv}, Brylawski defined the Tutte-Grothendieck ring in order to study functions invariant under certain types of decompositions. In the category with objects multisets of isomorphism classes of matroids and morphisms these matroid decompositions, the polynomial defined by a universal property coincides with the Tutte polynomial as defined earlier by Crapo. In this context, the Tutte-Grothediek invariants satisfy two fundamental recursions with respect to taking matroid minors.

Since the publication of \cite{JP18} and \cite{gorla2019rank,shiromoto19} shortly thereafter, there has been a good deal of interest in studying $q$-analogues of matroids and polymatroids \cite{BCJ_JCTB,GLJ,dir_sum_jur_cer,gorla2019rank,GLJindep,shiromoto19,IMAMURA,shellability}. A $q$-matroid is defined in the same way as a matroid: it comprises a modular lattice endowed with a rank function satisfying certain rank axioms. The concept of such generalisations of matroids were already evident in Crapo's early work \cite{crapo1964theory}. The main difference between the two objects is that while a finite matroid consists of a rank function defined on the lattice of subsets of a finite set, the underlying lattice in the case of a $q$-matroid is the lattice of subspaces of a finite-dimensional vector space. We have chosen to present our results to encompass both the matroid and $q$-matroid case. That is, for most of the paper, $M$ denotes either a matroid or a $q$-matroid, with the matroid case often emerging as a result of the specialization $q \longrightarrow 1$. 

A first attempt to define the Tutte polynomial of a $q$-matroid has been given in \cite{BCJ17}. Building on this preliminary work, in this paper we describe a construction of the Tutte polynomial based on an appropriate partition of the underlying support lattice into intervals that correspond to prime-free minors, which we call a Tutte partition. Somewhat surprisingly, we show that such partitions in the matroid case include the class of partitions arising in \cite{C69}, while not representing a direct $q$-analogue of such partitions the latter forming a strict subclass of the Tutte partitions defined here. Our description yields a Tutte polynomial both for matroids and $q$-matroids. This is our principal contribution, which is motivated by the need to address this fundamental gap in $q$-matroid theory. Having said that, our proof of the existence of a Tutte partition for a given $q$-matroid depends on the existence of a {\em proper interval decomposition} (defined in \cite{HITZEMANN20103551}), which is known to not exist for some classes of subspace lattices. We posit that the notion of a Tutte partition is an important step towards a definition of a Tutte polynomial that encompasses all $q$-matroids, which is a point we outline at the end of this paper with the help of an example.

We furthermore consider the notion of a Tutte-Grothendieck invariant. We define a new set of axioms and call the functions on the equivalence classes of matroids or $q$-matroids that satisfy this set of axioms of a $q$-Tutte-Grothendieck invariant. We show that in the case of a finite, Boolean lattice (the matroid case) these axioms, which are defined in reference to the concept of a matroid being prime-free or not prime-free, are equivalent to the classical axioms of a Tutte-Grothendieck invariant defined by Brylawski.

We consider the relation between the Tutte polynomial we define here and the rank generating polynomial and prove that one can be obtained from the other by convolution. We show that while the Tutte polynomial is a $q$-Tutte-Grothendieck invariant both in the matroid and $q$-matroid case, in the case of $q$-matroid, the rank generating polynomial is not.

This paper is organised as follows. In Section \ref{sec:prelim} we gather together basic results on lattices, matroidal $\{0,1\}$-weightings (bicolourings) and rank functions. In Section \ref{sec:q-Tutte_partition} we define the notion of a Tutte partition (in fact we give two equivalent definitions) and describe its properties. With this notion in place, we hence define a $q$-analogue of the Tutte polynomial with respect to a given Tutte partition. In Section \ref{sec:inv}, we give an inversion formula, which describes a change of basis in the polynomial ring $\Z[x,y]$. That is, we present a pair of mutually inversive maps $\alpha,\beta$ on $\Z[x,y]$. This answers a conjecture given in \cite{BCJ17}. In Section \ref{sec:ranktutte}, we show that the rank generating polynomial and any Tutte polynomial defined as in Section \ref{sec:q-Tutte_partition} can be retrieved from each other by applying the maps $\alpha$ or $\beta$ of Section \ref{sec:inv}, as appropriate. From this it follows that the Tutte polynomial defined with respect to a given Tutte partition is independent of the choice of Tutte partition and further, that we retrieve classical Tutte polynomial in the matroid case. For this reason, the Tutte polynomial we define may be referred to simply as the Tutte polynomial of the given matroid or $q$-matroid. In Section \ref{sec:q-T-G_invar} we consider the functions on the class of either matroids or $q$-matroids that satisfy certain recursion formulae with respect to minors. We introduce the notion of a $q$-Tutte-Grothendieck invariant, and give a set of axioms defining any such invariant. We show that while the Tutte polynomial of a $q$-matroid is a $q$-Tutte-Grothendieck invariant, (unlike the matroid case) the rank generating polynomial of a $q$-matroid is not. In Section \ref{sec:return_to_matroid_case}, we specialize to the matroid case. We first consider the Tutte partition introduced in Section \ref{sec:q-Tutte_partition} and compare this with the partition of the matroid into minors that arises in Crapo's description of the Tutte polynomial given in \cite{C69}, which we call a Crapo-Tutte partition. We show that while every Crapo-Tutte partition is also a Tutte partition, the converse does not hold: for matroids the Tutte partitions form a strictly larger class of partitions than those defined by Crapo, which extend Tutte's definition of the internal and external activity of the bases of a graph. We also compare the axioms of a $q$-Tutte-Grothendieck invariant with those of a Tutte-Grothedieck invariant \cite{brylawski_oxley_1992} and show that for matroids, both axiom systems are equivalent. Finally, in Section \ref{sec:furtherpropertise}, we consider some basic properties of the Tutte and rank generating polynomials and some alternative representations of them as linear maps. 

\section{Preliminaries}\label{sec:prelim}

We introduce some notation that will be used throughout this paper.

\begin{notation}
    Throughout this paper, we let $q$ denote a prime power and we let $\Fq$ denote the finite field of order $q$. For any positive integer $n$, we write $[n]:=\{1,\dots,n\}$.
\end{notation}
\begin{definition}
   Let $n,k$ be non-negative integers. Let $q \in \Z$.
    \begin{enumerate}
    \item If $q\neq 1$, then $[n]_q:=\frac{q^n-1}{q-1}$; if $q=1$, then $[n]_q:=n$.
    \item The $q$-factorial of $n$ is defined by $[n]_q!:=\displaystyle\prod_{k=1}^n[k]_q$. 
    \item The $q$-binomial coefficient is defined by 
    \[
        \qbin{n}{k}{q}:=\left\{
        \begin{array}{cc}
            \displaystyle\frac{[n]!_q}{[n-k]!_q[k]!_q} & \text{ if } k\leq n\\
            0 & \text{ if } k > n
        \end{array}
        \right.
    \]

\end{enumerate}
\end{definition}

The $q$-analogue of a binomial coefficient is called the $q$\emph{-binomial coefficient} or {\em Gaussian coefficient}; $\left[\begin{smallmatrix}n\\k \end{smallmatrix}\right]_q$ counts the number of $k$-dimensional subspaces of the vector space $\mathbb{F}_q^n$.
We recall the Pascal and $q$-Pascal identities for binomial and $q$-binomial coefficients respectively, and the $q$-binomial theorem.
These well-known identities, as well as the $q$-Binomial Theorem, can be found in texts such as \cite{qdist,stanley:1997}.
%%%%%%%%%%%%%%
\begin{theorem}
For $k,n\in\mathbb{Z}$, if $0\leq k\leq n$, the following identities hold:
\begin{align*}
    \binom{n}{k}&=\binom{n}{k-1}+\binom{n-1}{k-1},\\
    \left[\begin{matrix}n\\k \end{matrix}\right]_q&=q^k\left[\begin{matrix}n\\k-1 \end{matrix}\right]_q+\left[\begin{matrix}n-1\\k-1 \end{matrix}\right]_q,\\
    \left[\begin{matrix}n\\k \end{matrix}\right]_q&=\left[\begin{matrix}n\\k-1 \end{matrix}\right]_q+q^{n-k}\left[\begin{matrix}n-1\\k-1 \end{matrix}\right]_q.
\end{align*}
\end{theorem}

\begin{theorem}[The $q$-Binomial Theorem]
For $n,t\in\mathbb{Z}$, the following identity holds:
$$\displaystyle\prod_{k=0}^{n-1}(1-q^kt)=\sum_{k=0}^n(-1)^kq^{\binom{k}{2}}\left[\begin{matrix}n\\k \end{matrix}\right]_qt^k.$$
\end{theorem}
The following well-known theorem can be read, for example, in \cite[Theorem 7]{tingley91}.
\begin{theorem}\label{thm:tingley}
Let $V$ be an $n$-dimensional vector space over $\F_q$ and let $W$ be an $m$-dimensional subspace of $V$. Then $|\{U\leq V:U\oplus W=V\}|=q^{m(n-m)}$.
\end{theorem}

All of the above formulae and identities have classical counterparts, which can be retrieved by setting $q=1.$ In addition, we will use the following notation.

\begin{definition}\label{def:negbin}
    Let $n,k$ be nonnegative integers. Then
    $$\binom{-n}{k}:=\frac{\prod_{j=0}^{k-1}(-n-j)}{k!} = (-1)^k \binom{n+k-1}{k}.$$
\end{definition}

We fix notation and recall some basic definitions on lattices and posets. While it is beyond the scope of this paper to give a review of lattices to those who are new to the theory, we present some of the relevant terminology for the convenience of the reader. Throughout this paper, we will assume that all posets and lattices are finite, unless explicitly stated otherwise. The reader is referred to \cite{gratz_lattice,romanlattices} for further reading on lattices.

\begin{definition}
	Let $(P,\leq)$ be a poset. The poset $P$ is called a \emph{lattice} if every finite subset of $P$ has a well-defined \emph{meet} and \emph{join}.
	That is, for any finite subset $S$ of $P$ the following holds.
	There exists $b \in P$, called the meet or \emph{greatest lower bound} of $S$ in $P$ with the property that $b \leq a$ for all $a \in S$ and if there exists $c \in P$ such that $c \leq a$ for all
		$a \in S$ then $c \leq b$.
	There exists $b \in P$, called the join or \emph{least upper bound} of $S$ in $P$ with the property that $a \leq b$ for all $a \in S$ and if there exists 
		$c \in P$ such that $a \leq c$ for all
		$a \in S$ then $b \leq c$.
\end{definition}

\begin{definition}\label{def:iso}
	Let $(\mL_1,\vee_1,\wedge_1),(\mL_2,\vee_2,\wedge_2)$ be a pair of lattices. 
	A \emph{lattice isomorphism} between $\mL_1$ and $\mL_2$ is a bijective function $\varphi:\mL_1\longrightarrow\mL_2$ that is order-preserving and that preserves the meet and join; 
	that is, for all $x,y\in\mL_1$ we have that 
	\begin{enumerate}
		\item $x \leq_1 y \implies \varphi(x)\leq_2\varphi(y)$,
		\item $\varphi(x\wedge_1 y)=\varphi(x)\wedge_2\varphi(y)$,
		\item $\varphi(x\vee_1 y)=\varphi(x)\vee_2\varphi(y)$.
	\end{enumerate}	 
\end{definition}
	It is straightforward to check that reversing the ordering of a lattice yields a lattice, with the meet and join interchanged. 
	Combining this operation with a lattice isomorphism gives a lattice anti-isomorphism. 
\begin{definition}\label{def:anti}
		Let $(\mL_1,\vee_1,\wedge_1),(\mL_2,\vee_2,\wedge_2)$ be a pair of lattices. A \emph{lattice anti-isomorphism} between $\mL_1$ and $\mL_2$ is a bijective function $\psi:\mL_1\longrightarrow\mL_2$ that is order-reversing and that interchanges the meet and join; that is, for all $x,y\in\mL_1$ we have: 
		\begin{enumerate}
			\item $x \leq_1 y \implies \varphi(y)\leq_2\varphi(x)$,
			\item $\varphi(x\wedge_1 y)=\varphi(x)\vee_2\varphi(y)$,
			\item $\varphi(x\vee_1 y)=\varphi(x)\wedge_2\varphi(y)$.
		\end{enumerate}
\end{definition}

\begin{definition}
Let $\mathcal{L}$ be a lattice with meet $\wedge$ and join $\vee$. Let \textbf{0} and \textbf{1} denote the minimal and maximal element of $\mL$, respectively. 
Let $a, b\in\mathcal{L}$ such that $a \leq b$.
\begin{enumerate}
    \item 
    An {\em interval} $[a,b]\subseteq\mathcal{L}$ is the set of all $x\in\mathcal{L}$ such that $a\leq x\leq b$. It defines the {\em interval sublattice} $([a,b],\leq,\vee,\wedge)$. 

    \item Let $b\in  \mL$. We say that $b$ is a {\em complement} of $a$ if
    $a \wedge b = \textbf{0}$ and $a \vee b = \textbf{1}$.
    
    \item $\mL$ is called {\em relatively complemented} if every 
    $x \in [a,b]$ has a complement in $([a,b],\leq,\vee,\wedge)$ and is called {\em complemented} if every $x \in \mL$ has a complement in $\mL$.

    \item If $[a,b]\subseteq L$ such that $x\in[a,b]$ implies that $x=a$ or $x=b$ for any $x\in L$, then $b$ is called a {\em cover} of $a$ and we write $a\lessdot b$. We also say that $b$ covers $a$ or that $a$ is covered by $b$.
    If $a\lessdot b$, we also call the interval $[a,b]$ a \emph{cover} in $\mL$.
    We write $\mC(\mL):=\{[a,b] \subset \mL : a\lessdot b \}$ to denote the set of (interval) covers of $\mL$.
    
    \item An {\em atom} or {\em point} of $\mL$ is any element that is a cover of $\textbf{0}$. A {\em coatom} or {\em copoint} of $\mL$ is any element that is covered by \textbf{1}.  
    
    \item A {\em chain} from $a$ to $b$ is a sequence of the form
    $ a< x_1< x_2< \ldots < x_k=x $ with $x_j\in\mathcal{L}$ for $j=1,\ldots,k$. We say that the chain has length $k$. The chain is called {\em maximal} if 
    $ a\lessdot x_1\lessdot x_2\lessdot \ldots \lessdot x_k=x $. 
    We write $\ell([a,b])$ to denote is the maximum length of any chain from $a$ to $b$.
    
    \item Let $x\in \mathcal{L}$. The \emph{height} of $x$, which we denote by $\h(x)$, is the maximum length of any chain from \textbf{0} to
    $x$.
    
    \item The interval $[a,b]$ is called a \emph{diamond} if $\ell([a,b])=2$. 
    
    \item $\mL$ is called {\em modular} if for all $a,b,c \in \mL$, we have that
$a \geq c \implies (a \wedge b) \vee c = a \wedge (b \vee c).$
\end{enumerate}
\end{definition}

We remark that a cover is an interval of length one. We summarize some well-known facts on modular lattices; the reader is referred to \cite[Chapter 4]{romanlattices} for the details.

\begin{theorem}\label{th:modlatt}
   Let $\mL$ be a (finite) modular lattice and let $x,y,z \in \mL$. The following hold.
   \begin{enumerate}
       \item[\namedlabel{weakcovering}{{\rm (1)}}] If $x \lessdot y$, then $x \vee z \leq y \vee z$ and 
       $x \wedge z \leq y \wedge z$.
       \item[\namedlabel{pushupdown}{{\rm (2)}}] We have that $x \wedge y \lessdot x$ if and only if $y \lessdot x \vee y$.
       \item[\namedlabel{JD}{{\rm (3)}}] All maximal chains from $x$ to $y$ have the same length (i.e. $\mL$ satisfies the Jordan-Dedekind property).
       \item[\namedlabel{modheight}{{\rm (4)}}] We have that $\h(a \vee b)+\h(a \wedge b) = \h(a)+\h(b)$.
   \end{enumerate}
\end{theorem}

Equivalent to Theorem \ref{th:modlatt} \ref{pushupdown}, $[x\wedge y,y]$ is said to \emph{project up} to $[x,x\vee y]$ and $[x,x\vee y]$ \emph{projects down} to $[x\wedge y,y]$.
We then write $[x\wedge y,y] \myup [x,x\vee y]$ and 
$[x,x\vee y] \mydn [x\wedge y,y]$.
We will use Theorem \ref{th:modlatt} \ref{pushupdown} repeatedly throughout this paper. 

Lemma \ref{lem:cover_or_eq} was shown in \cite[Proposition 2]{BCJ17}. We provide a proof for the convenience of the reader.
\begin{lemma}\label{lem:cover_or_eq}
Let $\mathcal{L}$ be a modular lattice and let $x,y,z\in \mathcal{L}$ such that $[x,y]$ is a cover in $\mathcal{L}$. Exactly one of the following holds:
\begin{enumerate}
    \item $x\vee z\lessdot y\vee z \textup{ and } y\wedge z=x\wedge z$, or
    \item $y\vee z=x\vee z \textup{ and } x\wedge z\lessdot y\wedge z.$
\end{enumerate}
\end{lemma}

\begin{proof}
   By Theorem \ref{th:modlatt}~\ref{modheight}, we have that
   $\h(y \vee z) - \h(x \vee z) = \h(y)-\h(x) -\h(y \wedge z) + \h(x \wedge z)$. By \ref{JD} of the same theorem, for any $a,b \in \mL$, $1 = \h(b)-\h(a)$ if and only if $a \lessdot b$. Therefore,
   $\h(y \vee z) - h(x \vee z) = 1 -\h(y \wedge z) + \h(x \wedge z)$. It follows that either 
   $x \vee z \lessdot y \vee z$ and $x \wedge z = y \wedge z$ or $x \vee z = y \vee z$ and $x \wedge z \lessdot y \wedge z$. 
\end{proof}

\begin{notation}
  Throughout this paper, $\mL$ will denote a modular lattice ordered under $\leq$ with meet $\wedge$ and join $\vee$. We will write $\mC$ to denote the set of covers in $\mL$. We write $\textbf{0}:=0_\mL$ and $\textbf{1}:=1_\mL$. A complement of $v\in\mL$ will be denoted by $v^c$. We will let $V$ denote an $\F_q$-vector space of dimension $n$. We will write $\mL(V)$ to denote the lattice of subspaces of $V$, ordered by inclusion. 
\end{notation}

For the objects under consideration in this paper, one of the most relevant lattices will be the lattice of subspaces the finite-dimensional vector space $V$.
The lattice $\mL(V)$ has meet operation given by intersection and has join operation given by the vector space sum. The height of an element in this lattice is its dimension as a subspace of $V$. For $x\leq y$, the length of the interval $[x,y]$ is the dimension of the quotient space $y/x$ and $[x,y]$ is a cover if $x$ has codimension 1 in $y$.
If $S$ is a finite set, we write $\mL(S)$ to denote the Boolean lattice with greatest element $S$. The height of an element of $\mL(S)$ is given by its cardinality and for $x\leq y$ the length of the interval $[x,y] \subseteq \mL(S)$ is $|x - y|$.

There are numerous ways to define a $q$-matroid, \cite{BCJ_JCTB,JP18}. We will first define the concept of a $q$-matroid as a $\{0,1\}$-{\em weighting} of a lattice. This is essentially the same concept as a {\em bicolouring}, as described in \cite{BCJ17}.

\begin{definition}\label{def:lat_q-matroid}
A $\{0,1\}$-weighting of $\mathcal{L}$ 
is a map $w:\mC\longrightarrow \{0, 1\}$. We say that $\w([a,b])$ is the {\em weight} of the cover $[a,b] \in \mC$. 
For any maximal chain 
$ x_0\lessdot x_1\lessdot x_2\lessdot \ldots \lessdot x_k $, we define
the weight of the chain to be $\sum_{j=0}^{k-1} \w([x_i,x_{i+1}])$.
If $[a,b]$ is a diamond of $\mL$, we say that it has has type {\em full, empty, mixed} or {\em prime} with respect to $\w$ if it satisfies one of the following descriptions.
\begin{description}
    \item[Full:] $\w([u,v])=1$ for each $[u,v] \in \mC([a,b])$.
    \item[Empty:] $\w([u,v])=0$ for each $[u,v] \in \mC([a,b])$.
    \item[Mixed:] There exists $x \in [a,b],x\neq a,b$ such that 
    $\w([x,b])=1$ and  $\w([a,x])=0$, while 
    $\w([y,b])=0$, and $\w([a,y])=1$ for every $y \in [a,b],y \neq x,a,b$.
    \item[Prime:] $\w([a,x])=1$ and $\w([x,b])=0$ for any $x\in[a,b],x\neq a,b$.
\end{description}
\end{definition}

\begin{figure}[h]
    \centering
    \begin{tabular}{c|c|c|c}
    \hline &&&\\
      \includegraphics[scale=0.82]{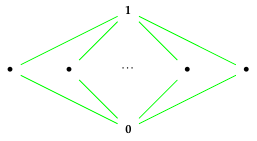}&
      \includegraphics[scale=0.82]{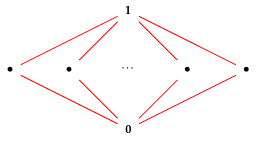}&
       \includegraphics[scale=0.82]{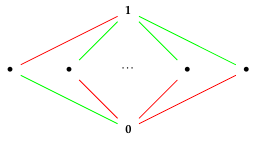}&
        \includegraphics[scale=0.82]{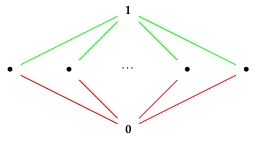}\\
        empty & full & mixed & prime\\
        &&&\\
        \hline
    \end{tabular}
    \caption{The four matroidal diamonds.}
    \label{fig:matroidaldiamonds}
\end{figure}

\begin{definition}\label{def:matweighting}
   A $\{0,1\}$-weighting $\w$ of $\mathcal{L}$ is called {\em matroidal} if each diamond in 
   $\mathcal{L}$ is of type full, empty, mixed, or prime.
\end{definition}

The following lemma can be  shown by applying an inductive argument to Definition \ref{def:matweighting}.
\begin{lemma}[\cite{BCJ17}]\label{lem:green_red_proj}
Let
$[a,b],[c,d]$ be covers in $\mL$, endowed with a matroidal $\{0,1\}$-weighting. The following hold:
\begin{enumerate}
    \item If $\w([a,b])=0$ and $[a,b] \myup [c,d]$ then $\w([c,d])=0$.
    \item If $\w([a,b])=1$ and $[a,b] \mydn [c,d]$ then $\w([c,d])=1$.
\end{enumerate}
\end{lemma}

\begin{corollary}
    Let $\mathcal{L}$ be a modular lattice and let $x,y,z\in \mathcal{L}$ such that $[x,y]$ is a cover in $\mathcal{L}$. Then the following hold.
    \begin{enumerate}
    \item If $y\wedge z=x\wedge z$ and $\w([x,y])=0$ then $\w([x\vee z,y\vee z]) = 0$.
    \item If $y\vee z=x\vee z$ and $\w([x,y])=1$ then $\w([x\wedge z,y\wedge z]) = 1$.
\end{enumerate}
\end{corollary}
\begin{proof}
    The result immediately follows from Lemma~\ref{lem:cover_or_eq} and Lemma~\ref{lem:green_red_proj}.
\end{proof}

\begin{definition}\label{def:rank_q-matroid}
Let $r:\mathcal{L}\rightarrow\mathbb{Z}_{\geq0}$ be a function. We define the following {\em rank axioms}.
\begin{enumerate}
    \item[(R1)] $0\leq r(x)\leq \h(x)$ for all $x\in\mathcal{L}$.
    \item[(R2)] For all $x,y\in\mathcal{L}$, $r(x)\leq r(y)$ whenever $x\leq y$.
    \item[(R3)] For all $x,y\in\mathcal{L}$, $r(x \vee y)+r(x \wedge y)\leq r(x)+r(y)$.
\end{enumerate}
We say that a rank function is {\em matroidal} if it satisfies the rank axioms (R1)-(R3).
\end{definition}

\begin{definition}
  Let $\mL$ be a lattice and $r:\mathcal{L}\rightarrow\mathbb{Z}_{\geq0}$ be a matroidal rank function. 
  The lattice $\mL$ is called the \emph{support lattice} of $M=(\mL,r)$.
  The elements of $\mathcal{L}$ are called the \emph{support elements} of $M$.
\end{definition}

Note that if $\mL$ is the Boolean lattice $\mL(S)$ for a finite set $S$ then $(\mL(S),r)$ is a matroid, while if $\mL=\mL(V)$, then $(\mL(V),r)$ is a $q$-matroid. 

\begin{definition}
Let $M=(\mL,r)$ be a matroid or a $q$-matroid. Let $a,b \in \mL$ such that $a\leq b$.
The function $r_{[a,b]}:[a,b]\rightarrow\mathbb{Z}$ is defined by $r_{[a,b]}(x)=r(x)-r(a)$ for all $x \in [a,b]$. 
The function $r_{[a,b]}$ is a matroidal rank function on $[a,b]\subseteq\mL$ and $([a,b],r_{[a,b]})$
is called a \emph{minor} of $M$.
\end{definition}

\begin{notation}
  Let $M=(\mL,r)$ be a matroid or $q$-matroid and let $a,b\in\mL$ with $a\leq b$. 
  For $[a,b]\subseteq\mL$ we write $M([a,b])$ to denote the minor $([a,b],r_{[a,b]})$.
  We will also define
  $M(a):=M([\zero,a])$, and define 
  $M/a:=M([a,\one])$. If $M$ is a matroid, we use $M-a$ to denote the restriction of $M$ to $[{\bf0},a^c]\subseteq\mL$.
\end{notation}

\begin{definition}
%Let $M$ be a $q$-matroid 
Let $r$ be a matroidal rank function on a support lattice $\mathcal{L}$. For any $x\in\mathcal{L}$, the \emph{nullity} of $x$ is
defined to be $\nu(x)=\h(x)-r(x)$. For an interval $[a,b]\subseteq\mathcal{L}$ %we define 
the function $\nu_{[a,b]}:[a,b]\rightarrow\mathbb{Z}$ %as the function defined by 
is defined by $\nu_{[a,b]}(x)=\nu(x)-\nu(a)$ for all $x\in[a,b]$.
\end{definition}

\begin{definition}
Let $M=(\mL,r)$ be a matroid or a $q$-matroid and let $x\in\mL$. If $\h(x)=1$ and $r(x)=0$, then $x$ is called a \emph{loop} of $M$. 
If $x$ is a coatom of $\mL$ such that $r(x)=r(M)-1$, then $x$ is called a \emph{coloop} of $M$. If there is an atom $e\in\mathcal{L}$ such that $e\leq y$ for any $y\in\mathcal{L}$ such that $r(y)=r(\one)$, then $e$ is called an \emph{isthmus} of $M$.
\end{definition}

\begin{remark}
     In the case of a matroid, a coloop is the dual of an isthmus in the dual matroid. However, this connection does not hold for $q$-matroids. Moreover, as shown in \cite[Theorem 5.4]{JPR22}, no $q$-matroid other than one whose rank equals its height contains an isthmus. Therefore, while the notion of an isthmus is frequently used in the context of the Tutte polynomial in matroid theory, we do not use it in $q$-matroid theory; instead we use the notion of a coloop. 
\end{remark}

A cryptomorphism between $q$-matroids defined via a rank function or via a $\{0,1\}$-weighting on the subspace lattice was shown in \cite[Theorem 2]{BCJ17}. It's easy to see that
the same result holds for matroids. We hence have the following statement.

\begin{theorem}[\cite{BCJ17}]
{\bf{~}}
   \begin{enumerate}
   \item
      Let $r:\mathcal{L}\rightarrow\mathbb{Z}_{\geq0}$ be a matroidal rank function. 
      Define a map $\w_r$ on the set of covers $\mC$ of $\mL$ by $\w_r([a,b]): = r(b)-r(a)$
      for all $[a,b] \in \mC$. Then $\w_r$ is a matroidal $\{0,1\}$-weighting of $\mL$.
   \item
      Let $\w:\mC \longrightarrow \{0,1\}$ be a matroidal $\{0,1\}$-weighting.
      Define a map $r_{\w}:\mathcal{L}\rightarrow\mathbb{Z}_{\geq 0}$ by
      $r_{\w}(x) := \sum_{i=1}^{t} \w([x_{i-1},x_i])$, where 
      ${\bf0}=x_0\lessdot x_1 \lessdot \cdots \lessdot x_{t}=x$ is a maximal chain in $[{\bf0},x]$. Then $r_{\w}$ is a matroidal rank function.
   \item
      Furthermore, $\w_{r_{\w}} = \w$ and $r_{\w_r}=r$.
   \end{enumerate}
\end{theorem}

\begin{definition}
Let $r$ be a matroidal rank function with support lattice $\mL$. 
 \begin{enumerate}
     \item For any $v\in\mL$, we define the \emph{closure} of $v$ with respect to $r$ to be $\bigvee\{e\in\mL:\h(e)=1, r(v\vee e)=r(v)\}$, and we denote it by $\textup{cl}(v)$.
     \item For any $v \in \mL$, we say that $v$ is {\em independent} with respect to $r$ if 
     $\h(v)=r(v)$. Otherwise we way that $v$ is {\em dependent}.
     \item For any $v \in \mL$, we say that $v$ is a {\em basis} with respect to $r$ if 
     $v$ is independent and $r(v)=r(\one)$.
 \end{enumerate}

\end{definition}
%%%%%%%%%%%%%%%%%%%%%%%%%%%%%%%%%%%%%%%%%%%%%%%%%%%%

\begin{definition}\label{def:lat-equiv}
Let $r_1$ and $r_2$ be a pair of matroidal rank functions with support lattices $\mL_1$ and $\mL_2$, respectively. We say that $r_1$ and $r_2$ are \emph{lattice-equivalent} if there exists a lattice isomorphism $\varphi:\mL_1\longrightarrow \mL_2$ such that $r_1(x)=r_2(\varphi(x))$ for all $x\in \mL_1$. In this case we write $M_1 \cong M_2$, where
$M_1=(\mL_1,r_1)$ and $M_2=(\mL_2,r_2)$.
\end{definition}

\begin{definition}
Let $M=(\mL,r)$ be a matroid or $q$-matroid and let $\w$ be the corresponding matroidal $\{0,1\}$-weighting of $\mL$. 
Let $\varphi$ be a lattice anti-isomorphism of $\mL$.
We write $r^*$ to denote
the matroidal rank function on $\varphi(\mL)$, which is defined by
\[ 
  r^*(\varphi(x)):=\h(\one)-\h(x)-(r(\one)-r(x)),
\]  
for all $x\in \mL$. 
We write $\w^*$ to denote the corresponding matroidal $\{0,1\}$-weighting of $\varphi(\mL)$. The \emph{dual} of $M$ is the matroid or $q$-matroid $(\varphi(\mL),r^*)$ and is denoted by $M^*$.

\end{definition}

It is straightforward to check that if $\w$ is a matroidal $\{0,1\}$-weighting of $\mL$ then its dual satisfies $\w^*([\varphi(y),\varphi(x)])=1-\w([x,y])$ for every cover $[x,y] \in \mC$. Observe also that $r^*(\varphi(x)) = \nu(\one)-\nu(x)$ for all $x \in \mL$.

\begin{remark}
  The definition of the dual of a matroidal rank function $r$ on $\mL$
  depends on the choice of anti-isomorphism $\varphi$ on $\mL$. However, it is clear that all such duals are lattice-equivalent.
\end{remark}

\begin{remark}
If $S$ is a finite set then there is a unique anti-automorphism induced by complementing, namely, $\varphi:\mL(S) \longrightarrow \mL(S): x \mapsto x^c$. For $\mL=\mL(V)$, any non-degenerate bilinear form on $V$ yields an anti-automorphism of $\mL(V)$ by the map that assigns $x \in \mL(V)$ to its orthogonal complement with respect to that bilinear form.  
\end{remark}

In several instances throughout this paper, it will be convenient to use the notion of duality typically used in matroid theory. 
 We recall the following lemma (c.f. \cite{byrne2019assmus}, \cite{JP18}, and \cite{oxley}). 

\begin{lemma}
    Let $\varphi:\mL \longrightarrow \mL$ be a lattice anti-automorphism. For each $a \in \mL$,
    let $\phi_a$ be an anti-automorphism of the interval $[0,\varphi(a)]$.
    Then $\phi_a \circ \varphi : [a,1] \longrightarrow [0,\varphi(a)]$ is a lattice isomorphism.
    If $r$ is a matroidal rank function with support lattice $\mL$, then
    \[
       M([a,1])^* \cong M^*([0,\varphi(a)]) \text{ and }
       M^*([a,1]) \cong M([0,\varphi(a)])^*.
    \]
\end{lemma}

For the remainder of this paper, unless explicitly stated otherwise, $r$ will denote a fixed but arbitrary matroidal rank function on $\mL$. We will fix $M:=(\mL,r)$, so $M$ will be understood to be a matroid if $\mL=\mL(S)$ for the set $S$ of cardinality $n$ and will be understood to be a $q$-matroid in $\mL=\mL(V)$ for the $n$-dimensional $\Fq$-subspace $V$. In respect of this approach, where formulas involving $q$-binomial coefficients and $q$-factorial numbers arise, for the matroid case these should be interpreted as binomial coefficients and factorials, respectively, or those found by letting $q\rightarrow1$, as appropriate. It is our intention that this approach will result in a simpler and more unified presentation of the theory.

%%%%%%%%%%%%%%%%%%%%%%%%%%%%%%%%%%%%%%%%%

%%%%%%%%%%%%%%%%%%%%%%%%%%%%%%%%%%%%%%%%%%%%%%%%%%%%%%%%%%%%%%%%%%%%%%%%%%%%%%

\section{Tutte partitions}\label{sec:q-Tutte_partition}

The purpose of this section is to provide a definition of the Tutte polynomial that can apply to both matroids and $q$-matroids alike. This definition relies on the existence of a type of partition of the support lattice of a matroid or $q$-matroid, which we call a \emph{Tutte partition}. 
The main result of this section, Theorem~\ref{thm:T-comp_partition_exist}, gives that a Tutte partition always exists on on the support lattice of a matroid or $q$-matroid, as long as the support lattice has a {\em proper interval partition} \cite{HITZEMANN20103551}. This provides a generalisation of the definition of the Tutte polynomial given in \cite{C69}, and follows \cite{BCJ17} in its use of \emph{prime-free} minors.

\begin{definition}
   We say that $M$ is {\em prime-free} if $M$ has no prime diamonds in $\mL$.
\end{definition}

\begin{definition}\label{def:Tutte_partition}
Let $\mathcal{P}$ be a partition of the elements of $\mathcal{L}$ such that every element of $\mP$ is an interval. We say that $\mathcal{P}$ is a \emph{Tutte partition} of $\mL$ if the following properties hold for every $[a,b]\in\mathcal{P}$:
\begin{enumerate}
    \item $r(a)=\h(a)$ (i.e. $a$ is independent),
    \item $r(b)=r(\one)$ (i.e. $b$ is spanning),
    \item $[a,b]$ is prime-free.
\end{enumerate}
\end{definition}

\begin{example}
    In Figure~\ref{fig:q-mat_lattice} we have an instance of a $q$-matroid that is not prime-free, as the interval $[0,\langle e_2,e_3\rangle]$ (among others) is a prime diamond. An example of a Tutte partition of this $q$-matroid is $$\{[0,\langle e_1,e_2\rangle],[\langle e_2+e_3\rangle,E],[\langle e_1+e_2+e_3\rangle,\langle e_1+e_3,e_2\rangle],[\langle e_3\rangle,\langle e_1+e_2,e_3\rangle],[\langle e_1+e_3\rangle,\langle e_1,e_3\rangle]\}.$$
    The subspace $\langle e_2,e_3\rangle$ is the only coloop in this $q$-matroid. There are no loops in this $q$-matroid. A diagram of this partition is given in Figure~\ref{fig:q-mat_lattice_part}.
\end{example}

\begin{figure}[h]
    \centering
    \includegraphics[scale=1]{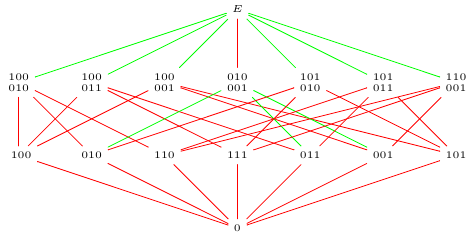}
    \caption{A $q$-matroid on $\mathbb{F}_2^3$. The notation $\begin{smallmatrix}100\\011\end{smallmatrix}$, for instance, denotes the subspace $\langle e_1,e_2+e_3\rangle$. The weight-1 covers have been coloured red, and the weight-0 covers have been coloured green.}
    \label{fig:q-mat_lattice}
\end{figure}

\begin{figure}[h!]
    \centering
    \includegraphics[scale=1]{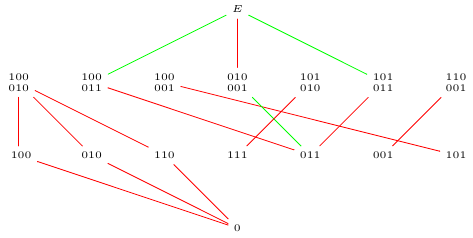}
    \caption{A Tutte partition of the $q$-matroid given in Figure~\ref{fig:q-mat_lattice}.}
    \label{fig:q-mat_lattice_part}
\end{figure}

Central to our strategy to define the Tutte polynomial of a $q$-matroid is the notion of a totally clopen element. Such an element only exists if $M$ is prime-free, in which case it determines the rank of every element of $M$. 

\begin{definition}
An element $z\in\mathcal{L}$  is called a {\em clopen} element of $M$ if for every $a\in \mL$ such that $z$ covers $a$ we have $\w([a,z]) = 0$ (i.e. $z$ is \emph{open}) and for every cover $b$ of $z$ we have $\w([z,b])=1$ (i.e. $z$ is \emph{closed}).
We say that $z$ is a \emph{totally clopen} element of $M$ if any cover $[x,y] \subset \mL$ satisfies $\w([x,y])=1$ if $x\geq z$, and $\w([x,y])=0$ if $y\leq z$.
\end{definition}

For example, the only clopen elements of the $q$-matroid in Figure~\ref{fig:q-mat_lattice} are $0$ and $\langle e_2,e_3\rangle$, and there is no totally clopen element in this $q$-matroid.

The following lemma follows immediately from the definition of a totally clopen element.

\begin{lemma}\label{lem:clopenfact}
   Let $z \in \mL$ be a totally clopen element of $M$. The following hold:
   \begin{enumerate}
       \item $r(z)=0$,
       \item $\nu(z)=\h(z)=\nu({\bf1})$,
       \item $\h(z)=\h(\one)-r(\one)$.
   \end{enumerate}
\end{lemma}

The following is a modification of \cite[Proposition 3]{BCJ17} that incorporates the notion of a totally clopen element, which we assert is a necessary concept for the statement to hold.

\begin{lemma}\label{lem:clopen_properties}
Suppose that $z \in \mL$ is a totally clopen element of $M$. Then for any cover $[x,y]$ in $\mL$ we have:
\[
   \w([x,y]) = \left\{ 
              \begin{array}{cc}
                 0  & \text{if } x\wedge z \lessdot y\wedge z,\\
                 1  & \text{if } x\vee z \lessdot y\vee z.
              \end{array}
   \right.
\]
Furthermore, $z$ is the unique totally clopen element of $M$ and the complements of $z$ are the bases of $M$.
\end{lemma}
\begin{proof}
Let $[x,y]$ be a cover in $\mathcal{L}$. By Lemma~\ref{lem:cover_or_eq}, either $[x\wedge z, y\wedge z]$ is a cover, or $[x\vee z, y\vee z]$ is a cover. If $[x\vee z, y\vee z]$ is a cover, then it has weight $1$ since $x\vee z\geq z$. Moreover, it is clear that $[x\vee z, y\vee z]\mydn[x,y]$ in this case, and so $\w([x,y])=1$. Similarly, if $[x\wedge z, y\wedge z]$ is a cover, then it has weight $0$ since $y\wedge z\leq z$, and since $[x\wedge z, y\wedge z]\myup [x,y]$, we have $\w([x,y])=0$. 

We now show that $z$ is the unique totally clopen element of $M$. Suppose that $w \in \mL$ is another totally clopen element of $M$. Consider that $\h(z)=\h(w)=\nu(\one)$, so if $w\neq z$, then $z\wedge w\lneq z$ and $z\wedge w\lneq w$. Therefore, there exists an atom $a\leq  z$ such that $a \nleq w$. Consider the cover $c=[z\wedge w,(z\wedge w)\vee a]$. Since $z$ is totally clopen and $(z\wedge w)\vee a \leq z$, we have $\w(c)=0$. 
On the other hand, as $w$ is totally clopen, we have $\w([w,a\vee w])=1$ and since $[w,a\vee w]\mydn [z\wedge w,(z\wedge w)\vee a]$, by Lemma \ref{lem:green_red_proj} it follows that $\w(c)=1$, which yields a contradiction.

Finally, to show that the complements of $z$ are the bases of $M$, first let $z^c$ be a complement of $z$. Then $z\wedge z^c={\bf0}$ and so for any cover $[x,y]\subseteq [{\bf0},z^c]$, we have $x\wedge z={\bf0}=y\wedge z$, which means, by Lemma \ref{lem:cover_or_eq}, that $[x\vee z,y\vee z]$ is a cover, and so $[x,y]$ has weight $1$. In particular, every cover of a maximal chain in $[{\bf0},z^c]$ has weight $1$ and so $r(z^c)=\h(z^c)$, which means that $z^c$ is a basis of $M$. Conversely, let $b\in\mathcal{L}$ be a basis of $M$. Since $r(z)=0$, we have $z\wedge b={\bf0}$. From Lemma \ref{lem:clopenfact}, we have $\h(z)=\h(1)-r(1)=\h(1)-\h(b)$, and thus $\h(z\vee b)=\h(z)+\h(b)-\h(z\wedge b)=\h(z)+\h(b)=\h(1)$, which gives $z \vee b = \one$. It follows that $b$ is a complement of $z$.
\end{proof}

\begin{remark}
    Suppose that $z \in \mL$ is a totally clopen element of $M$. Then $z = \cl({\bf0})$. To see this, note the following.
    As $r(z)=0$ we have $z \leq \cl({\bf0})$.
    If $[x,y]$ is a cover in $\mL$ such that $x \geq \cl({\bf0})$ then as $z$ is totally clopen we have $\w([x,y])=1$. Clearly every cover in $[{\bf0},\cl({\bf0})]$ has weight 0 and so it follows that $\cl({\bf0})$ is totally clopen in $M$. By the uniqueness of $z$ established in Lemma~\ref{lem:clopen_properties}, we deduce that $z=\cl({\bf0})$.
    On the other hand, for arbitrary $M$, it is not the case that $\cl({\bf0})$ is necessarily totally clopen. In particular, in a prime diamond there is no totally clopen element (see the Figure \ref{fig:matroidaldiamonds}).
\end{remark}

We give a precise characterization of $M$ for which a totally clopen element exists.

\begin{theorem}\label{th:no_prime_from_clopen}
There exists a totally clopen element $z\in \mathcal{L}$ in $M$ if and only if $M$ is prime-free.
\end{theorem}
\begin{proof}
Let $z\in \mathcal{L}$ be totally clopen. Assume that there exists a prime diamond $D=[d_b,d_t]$ in $\mathcal{L}$. By Lemma~\ref{lem:clopen_properties} and Lemma~\ref{lem:cover_or_eq}, for any cover $[d_b,x]\subset D$, $d_b\wedge z=x\wedge z$. Similarly, for any cover $[x,d_t]\subset D$, $x\vee z=d_t\vee z$. By Lemma~\ref{lem:green_red_proj} %
this means that $[d_b\vee z,x\vee z]=[d_b\vee z,d_t\vee z]$ is a weight-$1$ cover and $[x\wedge z,d_t\wedge z]=[d_b\wedge z,d_t\wedge z]$ is a weight-$0$ cover. By the modular law on $\mathcal{L}$ we have that $[(d_b\vee z)\wedge d_t,(d_t\vee z)\wedge d_t]=[d_b\vee (z\wedge d_t),d_t]$. If $d_b\vee (z\wedge d_t)\lessdot d_t$, then $[d_b\vee z,d_t\vee z]\mydn[d_b\vee (z\wedge d_t),d_t]$, which would make $[d_b\vee (z\wedge d_t),d_t]$ a weight-1 cover in $D$ containing $d_t$, contradicting $D$ being prime. Therefore, $d_b\vee (z\wedge d_t)= d_t$. Similarly, $[d_b\vee (z\wedge d_b),d_b\vee (z\wedge d_t)]=[d_b,(d_b\vee z)\wedge d_t]$ implies that $d_b=(d_b\vee z)\wedge d_t$, since otherwise $[d_b,(d_b\vee z)\wedge d_t]$ would be a weight-$0$ cover in $D$ containing $d_b$, contradicting $D$ being prime. This gives that $d_b=(d_b\vee z)\wedge d_t=d_b\vee (z\wedge d_t)=d_t$, which is impossible, and hence we conclude that $M$ has no prime diamonds.

Conversely, suppose that $M$ has no prime diamonds. If $r(M)=\h(\mathcal{L})$, then ${\bf0}\in\mathcal{L}$ is totally clopen in $M$, so suppose otherwise. 
Let $z=\textup{cl}(\zero)$. We claim that $z$ is totally clopen. Since $z$ is closed, we have that $\textup{w}([z,x_1])=1$ whenever $z\lessdot x_1\in\mathcal{L}$. Fix such an $x_1$ that covers $z$ and let $x_1\lessdot x_2\in\mathcal{L}$. If $\textup{w}([x_1,x_2])=0$, then $[z,x_2]$ must be a prime diamond, which contradicts our assumption that $M$ is prime-free. Therefore, $\textup{w}([x_1,x_2])=1$. Since our choices of $x_1$ and $x_2$ were arbitrary, this process can be repeated until we get a maximal chain $z=x_0\lessdot x_1\lessdot\dots\lessdot x_k=\one$ such that $\textup{w}([x_{i-1},x_{i}])=1$ for $i=1,\dots,k$. Therefore, $r(M)=\textup{h}(\mathcal{L})-\textup{h}(z)$, which means that all covers $[a,b]\subseteq\mathcal{L}$ such that $z\leq a$ must have weight 1.\qedhere
\end{proof}

In order for a partition $\mathcal{P}$ of $\mathcal{L}$ to satisfy the properties of Definition~\ref{def:Tutte_partition}, we must have that all loops of $M$ are contained in an interval of the form $[\zero,b]\subseteq\mathcal{L}$ and all coloops of $M$ are contained in an interval of the form $[a,\one]\subseteq\mathcal{L}$. Otherwise, it would be impossible for $\mathcal{P}$ to satisfy all of the properties of a Tutte partition, as either a loop would be below an interval in $\mathcal{P}$, or a coloop would be above an interval in $\mathcal{P}$. We address these requirements with the following results.

\begin{lemma}\label{lem:meet_of_coloops}
\textcolor{white}{.}
\begin{enumerate}
    \item There exists $y \in \mL$ of rank zero such that every loop of $M$ is contained in $[\zero,y]$.
    \item There exists $y\in\mL$ such that $[y,{\bf1}]$ contains every coloop of $M$ and $r({\bf1})-r(y)=\h(\one)-\h(y) = \ell([y,{\bf1}])$.
\end{enumerate}

\end{lemma}

\begin{proof} Clearly, there exists $y \in \mL$ such that $r(y)=0$ and every loop of $M$ is contained in the interval $[\zero,y]$, i.e. $y=\cl(\zero)$.
   Let $\varphi$ be an anti-isomorphism of $\mL$ and let $x$ be the closure of the minimal element of $\varphi(\mL)$ in $M^*$ and let $y \in \mL$ such that $x=\varphi(y)$. If $a$ is loop of $M^*$ then $a \leq x$ and so if $b:=\varphi^{-1}(a) \geq y$, then $r(b) = r(\one)-(\h(\one)-\h(b))=r(\one)-1$ and hence $b$ is a coloop of $M$. That is, there is a one-to-one correspondence between the loops of $M^*$ contained in $x$ and the coloops of $M$ that contain $y$. Finally, $r(y)=r^*(x)-\h(\one)+r(\one)+\h(y) =r(\one)+\h(y)-\h(\one)$ and so the result follows.
\end{proof}

\begin{corollary}\label{cor:atom_coatom_loops_coloops}
    Suppose that $0<r(\one)<\textup{h}(\one)$.
    \begin{enumerate}
        \item There exists a coatom $c\in\mathcal{L}$ such that every loop of $M$ is contained in $[\zero,c]$.
        \item There exists an atom $a\in\mathcal{L}$ such that every coloop of $M$ is contained in $[a,\one]$.
    \end{enumerate}
\end{corollary}
\begin{proof}
    Since $r(\one)\neq0$, we have that $\textup{cl}(\zero)<\one$, which means that there exists a coatom $c$ such that $\textup{cl}(\zero)\leq c$. By duality, the result follows.
\end{proof}

\begin{definition}
For an interval $[a,b]\subseteq\mL$ we say that a chain $C(a,b)\subseteq[a,b]$ is a \emph{clopen chain} if there exists $c\in C(a,b)$ such that for any cover $[x,y]\subseteq[a,b]$, we have:
\[
   \w([x,y]) =  \left\{ \begin{array}{cc}
      1  &  \text{ if } x \geq c \\
      0  &  \text{ if } y \leq c 
   \end{array}
   \right.
\]
Furthermore, we say that the clopen chain $C(a,b)$ is \emph{extendable} over $\mL$ if there exists $x\in \mL$ such that either $[b,x]$ is a weight-$1$ cover, or $[x,a]$ is a weight-$0$ cover. Otherwise, we say that $C(a,b)$ is \emph{non-extendable} over $\mL$.
\end{definition}

Note that if $M$ is prime-free, then there exists a clopen chain $C({\bf0},{\bf1})\subseteq\mL$ such that the totally clopen element $c\in\mL$ is in $C({\bf0},{\bf1})$.
\begin{lemma}\label{lem:non-ext_clopen_red_bot}
Let $C(a,b)\subseteq\mL$ be a clopen chain. Then $C(a,b)$ is non-extendable over $\mathcal{L}$ if and only if for every cover $[c,d]$ in $\mathcal{L}$, if $b\leq c$ then $[c,d]$ has weight $0$, and if $d\leq a$ then $[c,d]$ has weight $1$.
\end{lemma}
\begin{proof}
If $C(a,b)$ is non-extendable over $\mL$, then each cover of the form $[b,x]\subset\mL$ has weight 0, and each cover of the form $[x,a]\subset\mL$ has weight 1. Let $S=\{x\in[b,\one]:\h(x)=\h(b)+1\}$. By the submodularity of any matroidal rank function, we have that $r_{[b,\one]}(\one)\leq \sum_{x\in S}r_{[b,\one]}(x)=0$, which implies that all covers $[c,d]$ with $b\leq c$ have weight 0. 
By Lemma~\ref{lem:green_red_proj}, since all covers of the form $[x,a]\subset\mL$ have weight 1 by assumption, we have that all covers $[c,d]$ with $d\leq a$ have weight 1.

Conversely, if all covers $[x,y]\subseteq\mL$ with $y\leq a$ have weight 1 and all covers $[x',y']\subseteq\mL$ with $b\leq x'$ have weight 0, then clearly all covers $[x,a]$ have weight 1  and all covers $[b,y']$ have weight 0. Therefore, $C(a,b)$ is non-extendable over $\mL$.
\end{proof}
\begin{lemma}\label{prop:[e,1]_with_[0,e]_red}
Let $M$ contain a prime diamond. There then exists an atom $a\in\mL$ and a minor $[a,{\bf1}]\subset\mathcal{L}$ such that  all coloops of $M$ are contained in $[a,{\bf1}]$ and $\w([{\bf0},a])=1$.
\end{lemma}
\begin{proof}
This statement clearly holds if $\h(\one)=2$. Now let $\h(\one)=n\in\mathbb{N}$ and assume the statement holds for all matroids of length $\leq n-1$. By Corollary~\ref{cor:atom_coatom_loops_coloops}, there exists an atom $a\in \mL$ such that all coloops of $M$ are contained in $[a,{\bf1}]$. If $[a,{\bf1}]$ is prime-free, then by Theorem~\ref{th:no_prime_from_clopen} there is a totally clopen element in $[a,{\bf1}]$, and thus there is a clopen chain $C(a,{\bf1})$. Since $M$ contains a prime diamond, it does not have a totally clopen lattice element, and thus $C(a,{\bf1})$ is non-extendable over $\mL$. Therefore, by Lemma~\ref{lem:non-ext_clopen_red_bot}, $\w([{\bf0},a])=1$. Assume then that $[a,{\bf1}]$ contains a prime diamond. The induction hypothesis then gives us a minor $[f,{\bf1}]\subset[a,{\bf1}]$ such that all coloops in $[a,{\bf1}]$, and thus in $\mL$, are contained in $[f,{\bf1}]$ and furthermore, $\w([a,f])=1$. If we assume that $\w([{\bf0},a])=0$, then $[{\bf0},f]$ is a mixed diamond (since it contains a weight-1 cover as well as a loop), so we can choose an $e'\in[{\bf0},f]$ such that $\w([{\bf0},e'])=1$, and clearly $[e',{\bf1}]$ contains all the coloops of $\mL$ since $[f,{\bf1}]\subset[e',{\bf1}]$.
\end{proof}

\begin{lemma}\label{lem:disjoint_Tutte_minors}
Let $M$ contain a prime diamond. There then exists an atom $a\in\mL$ and copoint $c\in\mL$ such that every loop of $M$ is contained in $[{\bf0},c]$, every coloop of $M$ is contained in $[a,{\bf1}]$, and $[{\bf0},c]\cap[a,{\bf1}]=\emptyset$.
\end{lemma}
\begin{proof}
By Corollary~\ref{cor:atom_coatom_loops_coloops} and Corollary~\ref{cor:atom_coatom_loops_coloops} we know that such an atom $a$ and copoint $c$ exist, so all that remains to be shown is that they can be chosen such that $a\nleq c$. By Lemma~\ref{prop:[e,1]_with_[0,e]_red} we can choose $a\in\mL$ such that $[{\bf0},a]$ has weight $1$. As observed in Lemma~\ref{lem:meet_of_coloops}, for $y=\cl(\zero)$, the minor $[{\bf0},y]$ contains all the loops of $M$ and clearly $a\nleq y$. If all copoints of $M$ containing $y$ also contain $a$, then their meet also contains $a$, but their meet is $y$, a contradiction. Therefore, there exists a copoint $c\in\mathcal{L}$ that contains $y$ and does not contain $a$, and the result follows.
\end{proof}

Now we give an alternative definition of a Tutte partition (referred to below as a Tutte-\emph{compatible} partition) of $M$. The reason we do this is for convenience in proving the main result of this section. As we'll shortly see, both definitions are in fact equivalent.
\begin{definition}\label{def:T-compatible_partition}

A \emph{Tutte-compatible partition} is a partition of $M$ into prime-free minors, such that for any minor $[a,b]\subseteq\mL$ in this partition, any clopen chain $C(a,b)\in[a,b]$ is non-extendable over $\mathcal{L}$.
\end{definition}

\begin{example}
    Consider the $q$-matroid given in Figure~\ref{fig:q-mat_lattice}, which is not prime-free. Recall that a Tutte partition of this $q$-matroid into prime-free minors is shown in Figure~\ref{fig:q-mat_lattice_part}. The cover $[\langle e_1,e_2\rangle,E]$ has weight zero, so the clopen chain $C(0,\langle e_1,e_2\rangle)$ is non-extendable over $\mathcal{L}$. By inspection, one can see that the maximal clopen chains of each interval of this partition are non-extandable over $\mathcal{L}$. Therefore, this Tutte partition is also a Tutte-compatible partition.
\end{example}

\begin{lemma}\label{lem:T-compat_iff_T-part}
A partition of $M$ is a Tutte partition if and only if it is a Tutte-compatible partition.
\end{lemma}
\begin{proof}
Let $\mathcal{P}$ be a partition of $\mL$. % 
If $\mathcal{P}$ is a Tutte partition of $M$, then for any $[a,b]\in\mathcal{P}$, we have that $r(a)=\h(a)$, $r(b)=r({\bf1})$, and $[a,b]$ is prime-free. Therefore, $[a,b]$ contains a lattice element $z$ that is totally clopen in $[a,b]$ and any chain of the form 
$a \lessdot \cdots \lessdot z \lessdot \cdots \lessdot b$ is totally clopen in $[a,b]$. 
Since $r(a)=\h(a)$ we have that $\w([c,d])=1$ for any cover $[c,d]$ of $\mL$ such that $d\leq a$. Since $r(b)=r({\bf1})$ we have that $\w([c,d])=0$ for any cover $[c,d]$ of $\mL$ such that $c\geq b$. In particular, by Lemma~\ref{lem:non-ext_clopen_red_bot}, any totally clopen chain in $[a,b]$ is non-extendable over $\mL$.

Conversely, suppose that $\mathcal{P}$ is a Tutte-compatible partition. Since for any $[a,b]\in\mathcal{P}$ there is a clopen chain $C(a,b)$, by Theorem \ref{th:no_prime_from_clopen} we have that $[a,b]$ is prime-free. Since $C(a,b)$ is non-extendable over $\mL$, by Lemma~\ref{lem:non-ext_clopen_red_bot}, the weight of any cover in $[\zero,a]$ is $1$ and the weight of and cover in $[b,\one]$. It follows that $r(a)=\h(a)$ and $r(b)=r({\bf1})$.
\end{proof}

%%%%%%%%%%%%%%%%%%%%%%%%%%%%%%%%%%%%%%%%%%%%%%%%%%%%%%%%%%%%%%%%%%%%%
%%%%%%%%%%%%%%%%%%%%%%%%%%%%%%%%%%%%%%%%%%%%%%%%%%%%%%%%%%%%%%%%%%%%%

We now describe a partition of $\mL$, called an proper interval decomposition. The definition of a proper interval decomposition comes from the following. The Goldman-Rota recurrence \cite{goldmanrota}:
\[
  G_{n+1} = 2 G_n +(q^n-1) G_{n-1},
\]
for the Galois numbers $\displaystyle G_n:=\sum_{j=0}^n \qbin{n}{j}{q}$,
has inspired the question as to whether this recurrence could be regarded as a consequence of the existence of a partition of the lattice of subspaces 
into two intervals of length $n-1$ and $q^n-1$ intervals of length $n-2$.
It has been shown in \cite[Theorem 1]{HITZEMANN20103551}, that such a partition is equivalent to the existence of particular linear forms. While there are partial results, a full classification of the subspace lattices that admit a proper interval decomposition is open. Nonetheless, our approach relies on the existence of such a lattice partition.

\begin{definition}[\cite{HITZEMANN20103551}]
Let $e\in\mL$ be an atom and let $e^c\in\mathcal{L}$ be a complement of $e$.
Let $\mP$ be a partition of $\mL$ of the form 
\[
   [e,{\bf1}]\:\dot{\cup}\: [{\bf0},e^c]\:  \dot{\cup} \: [e_1,\Bar{e}_1] \:\dot{\cup} \:\cdots\:\dot{\cup}\:[e_{q^{n-1}-1},\bar{e}_{q^{n-1}-1}],
\]
where for each $k$, $e_k \leq \bar{e}_k$ for an atom $e_k$ and a coatom $ \bar{e}_k$.
We say that $\mP$ is a {\em proper interval decomposition} of $\mL$
and denote such a partition by $\mathcal{P}(e)$.
\end{definition}
Note that $\mathcal{P}(e)$ is not necessarily unique for each atom $e$, but it is sufficient for our purposes to choose one such partition for a given $e$. Note that $\mathcal{P}(e)$ is a non-trivial partition of $\mL$ into intervals with a minimal number of elements.
\begin{remark}
In the case that $\mL$ is a Boolean lattice, we have the proper interval decomposition
$\mathcal{P}(e)=[e,{\bf1}]\:\dot{\cup}\: [{\bf0},e^c]$, which is a partition of $M$ into two minors, one of which contains $e$ and one that contains all elements in the unique complement of $e$.
\end{remark}
We will now use Definition~\ref{def:T-compatible_partition} and Lemma~\ref{lem:T-compat_iff_T-part} to prove the following theorem. More explicitly,  we will show that a proper interval decomposition 
$\mathcal{P}(e)$, with $e$ and $e^c$ selected as described in the statement of Lemma \ref{lem:disjoint_Tutte_minors}, has a refinement that is a Tutte-compatible partition.

\begin{theorem}\label{thm:T-comp_partition_exist}
    Let $\mL$ have a proper interval decomposition. Then $M$ has a Tutte partition.
\end{theorem}
\begin{proof}
If $\h(\one)=2$ the statement clearly holds (note that intervals may have a length of zero), so assume that it holds for all $M=(\mL,r)$ such that $\h(\one)\leq n-1$. Let $\h(\one) = n$. If $M$ is prime-free then there is nothing to show, so assume that $M$ contains a prime diamond. By Lemma~\ref{lem:disjoint_Tutte_minors}, we can choose an atom 
$e\in\mL$ and a proper interval decomposition $\mathcal{P}(e)$ such that all coloops of $M$ are in $[e,{\bf1}]$ and all loops of $M$ are in $[{\bf0},e^c]$. Let $\{[e_i,\bar{e}_i]: 1 \leq i \leq q^{n-1}-1\}$ be the set of remaining intervals of $\mathcal{P}(e)$. % 
Since the length of each of these minors is at most $n-1$, the induction hypothesis states that each of these minors has a Tutte partition, and hence by Lemma~\ref{lem:T-compat_iff_T-part} each element of each such partition has a non-extendable clopen chain. By construction, the covers $[e^c,{\bf1}],[\bar{e}_1,{\bf1}],\dots,[\bar{e}_k,{\bf1}]$ all have weight 0 and the covers $[{\bf0},e],[{\bf0},e_1],\dots,[{\bf0},e_k]$ all have weight 1. Therefore, if $C(x,y)\subset\mL$ is a clopen chain in one of the intervals $A\in\mathcal{P}(e)$ and is non-extendable over $A$, it must be that $C(x,y)$ is also non-extendable over $\mL$ by Lemma~\ref{lem:non-ext_clopen_red_bot} since $r(x)=\h(x)$ and $r(y)=r({\bf1})$. The result now follows.
\end{proof}

\begin{notation}
  For an interval $A=[a,b]\subseteq\mL$, we write $\one_A:=b$ and $\zero_A:=a$.
\end{notation}

We are now ready to give a definition of a Tutte polynomial.
\begin{definition}\label{def:q-Tutte_poly}
Let $\mL$ have a proper interval decomposition. Let $\mathcal{P}$ be a Tutte partition of $M$. We define the {\em Tutte polynomial} of $M$ with respect to the partition $\mathcal{P}$ to be
$$\tau_{\mathcal{P}}(M;x,y)=\displaystyle\sum_{A\in\mathcal{P}}x^{r_A(\one_A)}y^{\nu_A(\one_A)}.$$
\end{definition}
%%%%%%%%%%%%%%%%%%%%%%%%%%%%%%%%%%%%%%%%%%%%%%%%%%%%%%%%%%%%%%%%%%%%%%%%%%%%%
While the definition we give here is given with respect to a partition $\mP$, we will see in Section~\ref{sec:ranktutte} that the Tutte polynomial is independent of the choice of Tutte partition.

\begin{example}\label{ex:P1tutte}
Let $M=(\mL(\F_2^3),r)$ be the $q$-matroid with support lattice shown on the left of Figure~\ref{fig:tutte_example}. On the right of Figure~\ref{fig:tutte_example}, we see a Tutte partition $\mP$ of $M$ into 5 prime-free minors. 
The diamond $[e,\one]=[\langle 101 \rangle, \one]$ is empty and has rank $0$ and nullity $2$, which contributes the term $y^2$ to $\tau_\mP(M,x,y)$. The minor with support lattice $[\zero,e^c]=[\zero,\langle 100,010\rangle]$ has rank $1$ and nullity $1$, which corresponds to the term $xy$. The remaining 3 intervals
support minors of length $1$, rank $0$, and nullity $1$.
\\~

\begin{figure}[h]
    \centering
    \includegraphics[scale=0.90]{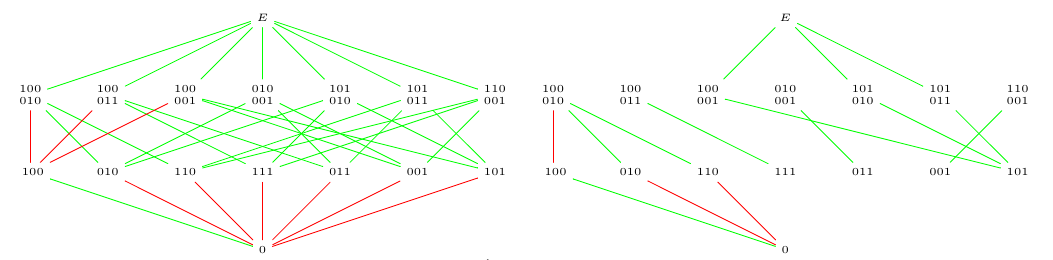}
    \caption{The Tutte Polynomial $\tau_\mP(M;x,y)=xy+y^2+3y$.}
    \label{fig:tutte_example}
\end{figure}
\end{example}

%%%%%%%%%%%%%%%%%%%%%%%%%%%%%%%%%%%%%%%%%%%%%%%%%%%%%%%%%%%%%%%%%%%%%%%%

\section{An inversion formula}\label{sec:inv}

We establish an inversion formula, which may be of independent interest. The main result of this section is Corollary~\ref{cor:alphabetainv}, which we will use in Section \ref{sec:ranktutte} to show that the rank generating polynomial of $M$ can be derived from a Tutte polynomial of $M$ and vice-versa.

\begin{definition}
    Let $n,m,i,j$ be non-negative integers. We define:
    \[
    \alpha_q(n,m;i,j):=\left[\begin{matrix}n\\i\end{matrix}\right]_q\left[\begin{matrix}m\\j\end{matrix}\right]_qq^{(n-i)(m-j)}.
    \]
\end{definition}

\begin{definition}
   Let $\gamma(i,j;a,b)$ be non-negative integers for all non-negative integers $i,j,a,b$.
   We define $\gamma : \Z[x,y] \longrightarrow \Z[x,y]$ by: 
   \[
   \gamma*P:=\sum_{i=0}^u \sum_{j=0}^v p_{i,j} \sum_{a=0}^{i}\sum_{b=0}^{j}\gamma(i,j;a,b)x^ay^b,
   \]
   for all $\displaystyle P(x,y)=\sum_{i=0}^u \sum_{j=0}^v p_{i,j} x^i y^j \in \Z[x,y]$. 
\end{definition}
%%%%%%%%%%%%%%%%%%%%%%%%%%%%%%%%%%%%%%%%%%%%%%%%%%%%%%%%

%%%%%%%%%%%%%%%%%%%%%%%%%%%%%%%%%%%%%%%%%%%%%%%%%%%%%%%

We will now establish the existence of a map 
$\beta: \Z[x,y] \longrightarrow \Z[x,y]$ such that 
$\alpha * \beta = \beta * \alpha$ is the identity map on $\Z[x,y]$. 
A candidate for such a map $\beta$ was suggested in \cite[Conjecture 1]{BCJ17} on the basis of experimental computations. We provide a proof of that conjecture in Theorem~\ref{thm:rank_to_tutte}.

\begin{definition}
    The \emph{Kronecker delta} is the binary function on $\mathbb{R}\times\mathbb{R}$ defined by:
    \begin{equation*}
        \delta^a_b=\left\{
        \begin{array}{cc}
            1 & \text{ if } a=b,\\
            0 & \text{ otherwise.}
        \end{array}
        \right.
    \end{equation*}
\end{definition}

\begin{definition}
Let $a,b,c,d$ be non-negative integers. Define:
$$\beta_q(a,b;c,d):=(-1)^{(a-c)+(b-d)}\left[\begin{matrix}a\\c\end{matrix}\right]_q\left[\begin{matrix}b\\d\end{matrix}\right]_qq^{\binom{|(a-c)-(b-d)|}{2}}(1+q^{|(a-c)-(b-d)|}-q^{\textup{max}(a-c,b-d)}).$$
\end{definition}
We will show that the following relation holds, showing that $\beta$ is the required inverse map of $\alpha$.
\begin{equation}\label{conj_1_BCJ}
    \displaystyle\sum_{c=0}^a\sum_{d=0}^b\alpha_q(a,b;c,d)\beta_q(c,d;u,v)=\delta_a^u\delta_b^v.
\end{equation}

We will use the following identity, which is an easy consequence of the $q$-binomial theorem. %(see \cite[Conjecture 2]{BCJ17}).
Recall Definition \ref{def:negbin} for the binomial coefficient, which appears in the following. 
\begin{lemma}\label{lem:BCJ_conj_2}
For any non-negative value of n, and any integer s with $0\leq s<n$,
$$\displaystyle\sum_{i=0}^{n}(-1)^iq^{\binom{i-s}{2}}\left[\begin{matrix}n\\i\end{matrix}\right]_q=\delta^n_0.$$
\end{lemma}
\begin{proof}
The $n=0$ case is trivial. Next observe that 
$$\binom{i-s}{2}=\frac{(i-s)(i-s-1)}{2}=\frac{i^2-is-i-is+s(s+1)}{2}=\binom{i}{2}-si+\binom{s+1}{2}$$
to get 
$$\displaystyle\sum_{i=0}^{n}(-1)^iq^{\binom{i-s}{2}}\left[\begin{matrix}n\\i\end{matrix}\right]_q=q^{\binom{s+1}{2}}\displaystyle\sum_{i=0}^{n}(-1)^iq^{\binom{i}{2}}\left[\begin{matrix}n\\i\end{matrix}\right]_qq^{-si}.$$
The $q$-binomial theorem states that for all $s$, we have
$$0=\displaystyle\prod_{i=0}^{n-1}(1-q^{i-s})=\displaystyle\sum_{i=0}^{n}(-1)^iq^{\binom{i}{2}}\left[\begin{matrix}n\\i\end{matrix}\right]_qq^{-si}.$$
The result now follows. 
\end{proof}
%%%%%%%%%%%%%%%%%%%%%%%%%%%%%%%%%%%%%%

Note that $\alpha_q(a,b;c,d)=0$ whenever $c>a$ or $d>b$, while $\alpha_q(a,b;a,d)=1=\alpha_q(a,b;c,b)$.
Also $\beta_q(c,d;u,v)=0$ whenever $u>c$ or $v>d$.
Therefore, we assume $c\leq a,u$ and $d\leq b,v$ and set $c=u+i$ and $d=v+j$ for non-negative integers $i,j$, which yields $\alpha_q(a,b;c,d)\beta_q(c,d;u,v)=$
\begin{align*}
    \left[\begin{matrix}a\\u\end{matrix}\right]_q\left[\begin{matrix}b\\v\end{matrix}\right]_q\left[\begin{matrix}a-u\\i\end{matrix}\right]_q\left[\begin{matrix}b-v\\j\end{matrix}\right]_q(-1)^{i+j}q^{(a-u-i)(b-v-j)}q^{\binom{|i-j|}{2}}(1+q^{|i-j|}-q^{\textup{max}(i,j)}).
\end{align*}
Since the summands in (\ref{conj_1_BCJ}) are now enumerated by $i$ and $j$, and
$\qbin{a}{u}{q}\qbin{b}{v}{q}$ 
is a nonzero common factor, we can quotient out this common factor without loss of generality. %
\begin{definition}
Let $a,b,c,d,u,v$ be non-negative integers satisfying $u\leq a$ and $v\leq b$. We define:
\[
N_{i,j}(a-u,b-v):=\alpha_q(a,b;c,d)\beta_q(c,d;u,v)\left(\left[\begin{matrix}a\\u\end{matrix}\right]_q\left[\begin{matrix}b\\v\end{matrix}\right]_q\right)^{-1}.
\]
\end{definition}

\begin{theorem}\label{thm:rank_to_tutte}

Let $a,b$ be non-negative integers. We have:
$$\displaystyle\sum^{a}_{i=0}\displaystyle\sum^{b}_{j=0}N_{i,j}(a,b)=\delta^a_0\delta^b_0.$$ 
\end{theorem}
\begin{proof}
Let $b>a$, then consider the $i$-th row sum
\begin{equation*}
\begin{aligned}
\displaystyle\sum^{b}_{j=0}N_{i,j}(a,b)=&(-1)^i\left[\begin{matrix}a\\i\end{matrix}\right]_q\displaystyle\sum^{b}_{j=0}(-1)^j\left[\begin{matrix}b\\j\end{matrix}\right]_qq^{(a-i)(b-j)}q^{\binom{|i-j|}{2}}(1+q^{|i-j|}-q^{\textup{max}(i,j)})\\
=&(-1)^i\left[\begin{matrix}a\\i\end{matrix}\right]_q(S_{1,i}+S_{2,i}-S_{3,i}),
\end{aligned}
\end{equation*}
where
\begin{equation*}
\begin{aligned}
S_{1,i}:=&\displaystyle\sum^{b}_{j=0}(-1)^j\left[\begin{matrix}b\\j\end{matrix}\right]_qq^{(a-i)(b-j)}q^{\binom{|i-j|}{2}},\\
S_{2,i}:=&\displaystyle\sum^{b}_{j=0}(-1)^j\left[\begin{matrix}b\\j\end{matrix}\right]_qq^{(a-i)(b-j)}q^{\binom{|i-j|}{2}}q^{|i-j|},\\
S_{3,i}:=&\displaystyle\sum^{b}_{j=0}(-1)^j\left[\begin{matrix}b\\j\end{matrix}\right]_qq^{(a-i)(b-j)}q^{\binom{|i-j|}{2}}q^{\textup{max}(i,j)}.
\end{aligned}
\end{equation*}
We will now express the $S_{t,i}$ in more convenient forms. First focusing on the exponents of $q$, we will consider the cases $i\geq j$ and $j\leq i$ separately.% to remove the absolute value signs. 

Suppose first that $i\geq j$ in $S_{1,i}$. We have
\begin{equation*}
\begin{aligned}
(a-i)(b-j)+\binom{i-j}{2}&=ab-aj-bi+ij+\frac{1}{2}(i-j)(i-j-1)\\
&=\frac{1}{2}(2ab-2aj-2bi+2ij+i^2-2ij-i+j^2+j)\\
&=\frac{(j-(a-1))(j-(a-1)-1)}{2}-\frac{a(a-1)}{2}+\frac{i(i-1)}{2}+b(a-i)\\
&=\binom{j-(a-1)}{2}-\binom{a}{2}+\binom{i}{2}+b(a-i).
\end{aligned}
\end{equation*}
For the case $i\leq j$ in $S_{1,i}$ we have:
\begin{equation*}
\begin{aligned}
(a-i)(b-j)+\binom{j-i}{2}&=ab-aj-bi+ij+\frac{1}{2}(j-i)(j-i-1)\\
&=\frac{1}{2}(2ab-2aj-2bi+2ij+i^2-2ij+i+j^2-j)\\
&=\frac{(j-a)(j-a-1)}{2}-\frac{a(a+1)}{2}+\frac{i(i+1)}{2}+b(a-i)\\
&=\binom{j-a}{2}-\binom{a+1}{2}+\binom{i+1}{2}+b(a-i).
\end{aligned}
\end{equation*}
For the exponents of $q$ in $S_{2,i}$ and $S_{3,i}$ we can just add the relevant term to the expressions above, the explicit computations of which are omitted, since it is clear at this point (eg. add $j-i$ to the last expression above to get the exponent part of $S_{2,i}$ when $i\leq j$). Now we see that we have
\begin{equation*}
\begin{aligned}
S_{1,i}&=q^{\binom{i}{2}-\binom{a}{2}+b(a-i)}\left(\displaystyle\sum^{i-1}_{j=0}(-1)^j\left[\begin{matrix}b\\j\end{matrix}\right]_qq^{\binom{j-(a-1)}{2}}\right)+q^{\binom{i+1}{2}-\binom{a+1}{2}+b(a-i)}\left(\displaystyle\sum^{b}_{j=i}(-1)^j\left[\begin{matrix}b\\j\end{matrix}\right]_qq^{\binom{j-a}{2}}\right),\\
S_{2,i}&=q^{\binom{i+1}{2}-\binom{a+1}{2}+b(a-i)}\left(\displaystyle\sum^{i-1}_{j=0}(-1)^j\left[\begin{matrix}b\\j\end{matrix}\right]_qq^{\binom{j-a}{2}}\right)+q^{\binom{i}{2}-\binom{a}{2}+b(a-i)}\left(\displaystyle\sum^{b}_{j=i}(-1)^j\left[\begin{matrix}b\\j\end{matrix}\right]_qq^{\binom{j-(a-1)}{2}}\right),\\
S_{3,i}&=q^{\binom{i+1}{2}-\binom{a}{2}+b(a-i)}\left(\displaystyle\sum^{b}_{j=0}(-1)^j\left[\begin{matrix}b\\j\end{matrix}\right]_qq^{\binom{j-(a-1)}{2}}\right)
=0.
\end{aligned}
\end{equation*}
Now add $S_{1,i}$ and $S_{2,i}$ and apply Lemma \ref{lem:BCJ_conj_2} to obtain:
\begin{equation*}
\begin{aligned}
S_{1,i}+S_{2,i}&=q^{\binom{i}{2}-\binom{a}{2}+b(a-i)}\sum^{b}_{j=0}(-1)^j\left[\begin{matrix}b\\j\end{matrix}\right]_qq^{\binom{j-(a-1)}{2}}+q^{\binom{i+1}{2}-\binom{a+1}{2}+b(a-i)}\sum^{b}_{j=0}(-1)^j\left[\begin{matrix}b\\j\end{matrix}\right]_qq^{\binom{j-a}{2}}\\
&=0.
\end{aligned}
\end{equation*}

It now follows that:
\begin{equation*}
\begin{aligned}
\displaystyle\sum^{a}_{i=0}\displaystyle\sum^{b}_{j=0}N_{i,j}(a,b)&=\displaystyle\sum^{a}_{i=0}(-1)^i\left[\begin{matrix}a\\i\end{matrix}\right]_q(S_{1,i}+S_{2,i}-S_{3,i})=0.
\end{aligned}
\end{equation*}
Since relabelling makes the above proof work for $a>b$, all that remains to check is when $a=b$. Now set $a=b$, with $S_{1,i}, S_{2,i}, S_{3, i}$, defined as before 
 (writing $a$ instead of $b$)
\begin{equation*}
\begin{aligned}
S_{1,i}+S_{2,i}&=q^{\binom{i+1}{2}-\binom{a+1}{2}+a(a-i)}\displaystyle\sum^a_{j=0}(-1)^j\left[\begin{matrix}a\\j\end{matrix}\right]_qq^{\binom{j-a}{2}}\\
&=q^{\binom{i+1}{2}-\binom{a+1}{2}+a(a-i)}f(a),\\
S_{3,i}&=0.
\end{aligned}
\end{equation*}
where $f(a)=\displaystyle\sum^a_{j=0}(-1)^j\left[\begin{matrix}a\\j\end{matrix}\right]_qq^{\binom{j-a}{2}}$. Putting things back together, using Lemma~\ref{lem:BCJ_conj_2}, we have
\begin{equation*}
\begin{aligned}
\displaystyle\sum^a_{i=0}\displaystyle\sum^a_{j=0}N_{i,j}(a, a)&=%0
\displaystyle\sum^a_{i=0}(-1)^i\left[\begin{matrix}a\\i\end{matrix}\right]_qq^{\binom{i+1}{2}-\binom{a+1}{2}+a(a-i)}f(a)\\
&=q^{a-1}f(a)\displaystyle\sum^a_{i=0}(-1)^i\left[\begin{matrix}a\\i\end{matrix}\right]_qq^{\binom{i-(a-1)}{2}}\\
&=0,
\end{aligned}
\end{equation*}
where the second to last equality comes from similar manipulations of the exponent of $q$ as was done for the $a<b$ case, and the last equality coming from an application of Lemma \ref{lem:BCJ_conj_2}.
\end{proof}

We now immediately obtain the following corollary.

\begin{corollary}\label{cor:alphabetainv}
   Let $a,b,c,d,u,v$ be non-negative integers. Then
   \begin{equation*}
    \displaystyle\sum_{c=0}^a\sum_{d=0}^b\alpha_q(a,b;c,d)\beta_q(c,d;u,v)=
    \delta_{a}^{u}\delta_{b}^{v}.%\delta_b^f.
    \end{equation*}
    In particular, for any $P \in \Z[x,y]$, we have
    \[
       \alpha_q * (\beta_q * P) = P = \beta_q* (\alpha_q * P).
    \]

\end{corollary}
%%%%%%%%%%%%%%%%%%%%%%%%%%%%%%%%%%%%%%%%%%%%%%%%%
\begin{remark}
The mutually inverse transformations $\alpha$ and $\beta$ are indeed $q$-analogues of the classical transformations associated with the rank generating polynomial and Tutte polynomial of a matroid.
Let $\displaystyle P(x,y)=\sum_{i=0}^u \sum_{j=0}^v p_{i,j} x^i y^j \in \Z[x,y]$. 
If we set $q=1$ in $\alpha_q$ and $\beta_q$ we obtain:
\[
    \alpha_1(i,j;a,b):=\binom{i}{a}\binom{j}{b}
    \text{ and }
    \beta_1(i,j;a,b):=(-1)^{i-a}(-1)^{j-b}\binom{i}{a}\binom{j}{b}.
    \]
    Therefore,
   \[
   \alpha_1*P:=\sum_{i=0}^u \sum_{j=0}^v p_{i,j} \sum_{a=0}^{i}\sum_{b=0}^{j} \binom{i}{a}\binom{j}{b}x^ay^b =\sum_{i=0}^u \sum_{j=0}^v p_{i,j}(x+1)^a(y+1)^b.
   \]
    and 
   \[
   \beta_1*P:=\sum_{i=0}^u \sum_{j=0}^v p_{i,j} \sum_{a=0}^{i}\sum_{b=0}^{j} (-1)^{i-a}(-1)^{j-b}\binom{i}{a}\binom{j}{b}x^ay^b =\sum_{i=0}^u \sum_{j=0}^v p_{i,j}(x-1)^a(y-1)^b.
   \]  
\end{remark}

%%%%%%%%%%%%%%%%%%%%%%%%%%%%%%%%%%%%%%%%%%%%%%%%%%%%%%%%%%%%%%%%%%%%%%%%%%%%

\section{The rank generating polynomial and the Tutte polynomial}\label{sec:ranktutte}

One of the fundamental properties of the Tutte polynomial in classical matroid theory is its connection to the rank generating polynomial. In this section we establish the connection between the rank generating polynomial of $M$ and $\tau_{\mP}(M;x,y)$ for any Tutte partition $\mP$. Explicitly, we will show that the rank generating polynomial is given by $\alpha_q * \tau_{\mP}(M;x,y)$ and conversely, any $\tau_{\mP}(M;x,y)$ can be derived by applying the transformation $\beta$ to the rank generating polynomial of $M$. This leads to another important result of this section, Corollary~\ref{cor:T-poly_indep_of_choice}, which gives that $\tau_{\mP}(M;x,y)$ is uniquely determined and independent of the choice of Tutte partition $\mP$.

\begin{definition}\label{def:rank_poly}
The {\em rank generating polynomial} 
of $M$ is defined to be 
$$\rho(M;x,y):=\displaystyle\sum_{z\in\mathcal{L}}x^{r({\bf1})-r(z)}y^{\nu(z)} \in \Z[x,y].$$
\end{definition}

We remark that a $4$-variable definition of the rank generating polynomial was given in \cite{shiromoto19}, where it was used to obtain a generalisation of Greene's theorem.
We will use the following result (see \cite[Equation (2)]{BCJ17}). We include the proof for completeness. This result will allow us to deduce that the rank generating polynomial of a prime-free matroid or $q$-matroid depends only on its rank and nullity, which is perhaps not surprising in the light of Lemma \ref{lem:clopen_properties}. Note that the \emph{rank-lack} of $x\in\mL$ is defined to be $r({\bf1})-r(x)$. We recall the notion of \emph{a relatively complemented} lattice, which means that for all intervals $[a,b]\subseteq\mathcal{L}$ and $x\in[a,b]$, there exists $y\in[a,b]$ such that $a=x\wedge y$ and $b=x\vee y$. It is well known that $\mathcal{L}$ is relatively complemented. This notion is used in the following proof.

\begin{proposition}[\cite{BCJ17}]\label{prop:p-free_elements}
Let $M$ be prime-free. Let $\h(\one)=n$, $r({\bf1})=\rho$, and $\nu({\bf1})=\nu$. The number of elements in $\mathcal{L}$ with rank-lack $i$ and nullity $j$ is $\alpha_q(\rho,\nu;i,j)$. 
\end{proposition}
\begin{proof}
    Let $c$ be the unique totally clopen element in $\mathcal{L}$, which has height $\nu=n-\rho$. Let $[a,b]\subseteq\mathcal{L}$ be such that $c\in[a,b]$ and let $\text{h}(a)=j\leq\nu$ and let $\text{h}(b)=n-i$ for $i\leq\rho$. Since $\mathcal{L}$ is relatively complemented, there exists $x\in[a,b]$ such that $a=x\wedge c$ and $b=x\vee c$. Since $\mathcal{L}$ is prime-free, we have that $[a,b]$ is prime-free with $c$ as its unique totally clopen element. By Lemma~\ref{lem:clopen_properties} we have that $x$ is a basis of the minor $M([a,b])$. Since $r(x\wedge c)=0$ and $\text{h}(x)-\text{h}(x\wedge c)=\text{h}(x\vee c)-\text{h}(c)$, we have that
    \begin{align}\label{eq:p-free_interval}
        r(x)&=r(x\wedge c)+r_{[x\wedge c,x\vee c]}(x)=r_{[x\wedge c,x\vee c]}(x)=\text{h}(x\vee c)-\text{h}(c)\nonumber\\
        &=n-i-\nu=n-i-(n-\rho)=\rho-i.
    \end{align}
    This means that $x$ has rank-lack $i$ in $M$. Clearly, $\nu(x)=\text{h}(x\wedge c)=j$.
    
    For any $x\in\mathcal{L}$, it is clear that $x$ is a relative complement of $c$ in the interval $[x\wedge c,x\vee c]$. Therefore, by (\ref{eq:p-free_interval}), the rank-lack and nullity of $x$ in $M$ are fully determined by $\text{h}(x\wedge c)$ and $\text{h}(x\vee c)$.

    To count the number of elements of $\mathcal{L}$ with rank-lack $i$ and nullity $j$, it thus suffices for us to compute 
    \[
    |\{(x,a,b):c \in [a,b], \h(a)=j, \h(b)=n-i, x \wedge c = a, x \vee c =b \}|.
    \]
    The number of elements $a\leq c$ with $\text{h}(a)=j$ is the number of elements of height $j$ in the interval $[\zero,c]$, which is $\qbin{\text{h}(c)}{j}{q}=\qbin{\nu}{j}{q}$. The number of elements $b\geq c$ with $\text{h}(b)=n-i$ is the number of elements of height $n-i-\text{h}(c)$ in the interval $[c,\one]$, which is $\qbin{n-\nu}{n-\nu-i}{q}=\qbin{\rho}{\rho-i}{q}=\qbin{\rho}{i}{q}$. 
    By Theorem~\ref{thm:tingley}, for any such interval $[a,b]$, the number of elements $x\in[a,b]$ such that $a=x\wedge c$ and $b=x\vee c$ is $q^{(\nu-j)(n-i-j-(\nu-j))}=q^{(\nu-j)(\rho-i)}$. Therefore, the number of elements in $\mathcal{L}$ with rank-lack $i$ and nullity $j$ is given by $\qbin{\rho}{i}{q}\qbin{\nu}{j}{q}q^{(\nu-j)(\rho-i)}$, which is equal to $\alpha_q(\rho,\nu;i,j)$.
\end{proof}

\begin{corollary}\label{cor:alpha_rank_poly_p-free}
Let $M$ be a prime-free and have rank $\rho$ and nullity $\nu$. Then the rank generating polynomial of $M$ is:
$$\rho(M;x,y)=\displaystyle\sum_{i=0}^{\rho}\sum_{j=0}^{\nu}\alpha_q(\rho,\nu;i,j)x^iy^j.$$
\end{corollary}

\begin{proposition}\label{thm:Tutte-to-rank_convo}\footnote{The statement of this proposition appeared as part of \cite[Theorem 3]{BCJ17}. However, the statement does not hold for arbitrary prime-free minors. We refer the reader to \cite[Section 10]{BCJ17} for a discussion on this matter.}
Let $\mathcal{P}$ be a Tutte partition of $M$ and let $\tau_{\mathcal{P}}(M;x,y)=\displaystyle\sum_{i=0}^{\rho}\sum_{j=0}^{\nu}\tau_{i,j}x^iy^j$. Then the rank generating polynomial of $M$ is:
\[
\rho(M;x,y)= \alpha_q * \tau_{\mP}(M;x,y).
\]
\end{proposition}

\begin{proof}
By definition, we have \[\tau_\mP(M;x,y)=\displaystyle\sum_{A\in\mathcal{P}}
x^{r_A(A)}y^{\nu_A(A)} = \sum_{i=0}^{\rho}\sum_{j=0}^{\nu}\tau_{i,j}x^iy^j,\] 
which means that $\tau_{i,j}$ denotes the number of elements of $\mathcal{P}$ with rank $i$ and nullity $j$. Let $A \in \mP$. By Corollary~\ref{cor:alpha_rank_poly_p-free}, we have that $\rho(A;x,y)=\displaystyle\sum_{a=0}^{r(A)}\sum_{b=0}^{\nu(A)}\alpha_q(r_A(A),\nu_A(A);a,b)x^ay^b$. 

Since $\mP$ is a Tutte partition of $M$, we have $r(\zero_A)=\h(\zero_A)$ and $r(\one_A)=r({\bf1})$. For any $z\in A$ we have that $r_{A}(\one_A)-r_{A}(z)=r({\bf1})-r(z)$ and $\nu_{A}(z)=\nu(z)$. Therefore,
\begin{align*}
    \rho(M;x,y)&=\displaystyle\sum_{z\in\mL}x^{r({\bf1})-r(z)}y^{\nu(z)}\\
    &=\displaystyle\sum_{A\in\mathcal{P}}\sum_{z\in A}x^{r({\bf1})-r(z)}y^{\nu(z)}\\
    &=\displaystyle\sum_{A\in\mathcal{P}}\sum_{z\in A}x^{r_A(\one_A)-r_A(z)}y^{\nu_A(z)}\\
    &=\displaystyle\sum_{A\in\mathcal{P}}\rho(A;x,y)\\
    &=\displaystyle\sum_{A\in\mathcal{P}}\sum_{a=0}^{r_A(\one_A)}\sum_{b=0}^{\nu_A(\one_A)}\alpha_q(r_A(\one_A),\nu_A(\one_A);a,b)x^ay^b\\
    &=\displaystyle\sum_{i=0}^{\rho}\sum_{j=0}^{\nu}\tau_{i,j}\displaystyle\sum_{a=0}^{i}\sum_{b=0}^{j}\alpha_q(i,j;a,b)x^ay^b.
\end{align*}
The result now follows.
\end{proof}
%%%%%%%%%%%%%%%%%%%%%%%%%%%%%%%%%%%%%%%%%

From Corollary \ref{cor:alphabetainv}, we obtain that for any Tutte partition $\mP$, $\tau_{\mP}(M;x,y)$ is determined by the rank generating polynomial of $M$.

\begin{corollary} \label{cor:T-poly_indep_of_choice}
Let $\mP$ be a Tutte partition of $M$.
Then 
\[
\beta_q * \rho(M;x,y) = \tau_{\mP}(M;x,y).
\]
In particular, the Tutte polynomial $\tau_{\mP}(M;x,y)$ of $M$ is independent of the choice of Tutte partition $\mP$.
\end{corollary}
\begin{proof}
From Proposition~\ref{thm:Tutte-to-rank_convo} and Corollary \ref{cor:alphabetainv}
we have
\begin{align*}
    \rho(M;x,y)&= \alpha_q * \tau_{\mP}(M;x,y),\\
    \implies \beta_q *\rho(M;x,y) &= \beta_q *\alpha_q * \tau_{\mP}(M;x,y) = \tau_{\mP}(M;x,y).
\end{align*}
That $\tau_{\mP}(M;x,y)$ is independent of the choice of $\mP$ now follows immediately by the uniqueness of the rank generating polynomial $\rho(M;x,y)$.
\end{proof}
Therefore, for the remainder of this paper, we dispense with the explicit reference to a Tutte-partition $\mP$ and simply write $\tau(M;x,y)$.
%%%%%%%%%%%%%%%%%%%%%%%%%%%%%%%%%%%%%%%%%%%%%%%%%%%%%%%

   If we specialize to the case for which $M$ is a matroid, we see that the Tutte polynomial as defined in this paper coincides with the Tutte polynomial as defined by Crapo in \cite{C69}, albeit defined in a different way. For this reason we will for the remainder of this paper, call $\tau(M;x,y)$ the Tutte polynomial of $M$, for both the matroid and $q$-matroid case. We will revisit the matroid case in Section \ref{sec:return_to_matroid_case}.

%%%%%%%%%%%%%%%%%%%%%%%%%%%%%%%%%%%%%%%%%%%%%%%%%%%%%%%%%%%%%

%%%%%%%%%%%%%%%%%%%%%%%%%%%%%%%%%%%%%%%%%%%%%%%%%%%%%%%%%%%%%%%%%%%%%%%%%%%%%%

\begin{example}\label{ex:Tutte_vs_rank_figures}
In Figure~\ref{fig:rank_and_tutte} we compute the Tutte polynomial and rank generating polynomial of some of the length-$3$ $q$-matroids listed in the appendix of \cite{dir_sum_jur_cer}.
 On the left hand side are the support lattices of the $q$-matroids, 
 assigned a bicolouring consistent with its matroidal $\{0,1\}$-weighting: weight-$0$ covers are coloured green and weight-$1$ covers are coloured red. 
 On the right-hand-side are corresponding Tutte partitions of each $q$-matroid. 
 Under each support lattice the rank generating polynomial of that $q$-matroid is written, and under each Tutte partition is its Tutte polynomial.
\vspace{0.1cm}

\begin{figure}[h]
    \centering
    \includegraphics[scale=1]{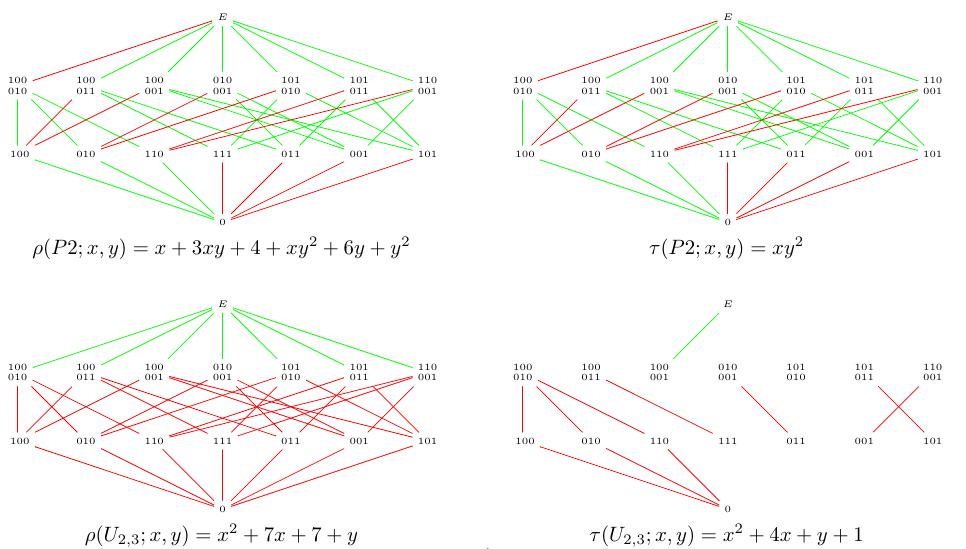}
    \caption{Rank and Tutte Polynomials}
    \label{fig:rank_and_tutte}
\end{figure}
\end{example}

%%%%%%%%%%%%%%%%%%%%%%%%%%%%%%%%%%%%%%%%%%%%%%%%%%%%%%%%%%
\begin{example}
Using the same naming convention as in Example~\ref{ex:Tutte_vs_rank_figures}, we have 
$$\rho(P2;x,y)=\displaystyle\sum_{a=0}^{\rho}\sum_{b=0}^{\nu}\rho_{a,b}x^ay^b=x+3xy+4+xy^2+6y+y^2.$$ By Theorem~\ref{thm:rank_to_tutte},
\begin{align*}
    \tau(P2;x,y)&=\displaystyle\sum_{a=0}^{\rho}\sum_{b=0}^{\nu}\rho_{a,b}\sum_{i=0}^a\sum_{j=0}^b\beta_q(a,b;i,j)x^iy^j\\
    &=(x-1)+3(-x-y+xy)+4+(1+2x-y^2-3xy+xy^2)\\
    &\hspace{1cm}+6(-1+y)+(2-3y+y^2)\\
    &=xy^2,
\end{align*}
which of course agrees with the lattice-theoretic computation of $\tau(P2;x,y)$.
\end{example}
%%%%%%%%%%%%%%%%%%%%%%%%%%%%%%%%%%%%%%%%%%%%%%%%%%%%%%%%%%%

%%%%%%%%%%%%%%%%%%%%%%%%%%%%%%%%%%%%%%%%%%%%%%%%%%%%%%%%%%

\section{$q$-Tutte-Grothendieck invariants}\label{sec:q-T-G_invar}

In the classical theory of matroids, the notion of a Tutte-Grothendieck (T-G) invariant is fundamental to the study of the Tutte polynomial \cite{brylawski_oxley_1992,brylowski_ams,brylowski_alguniv}. A T-G invariant is based upon a category theoretic construction by Brylawski, which is based upon a ring theoretic construction given by Tutte \cite{tutte_1947}. Such a construction cannot apply to $q$-matroids as it is too strict. The main contribution of this section is a weakening of this construction, which we call a \emph{$q$-T-G invariant}, in Definition~\ref{def:qTG} (in fact, we show in Section~\ref{subsec:T-G_invar} that in the case of matroids both constructions are equivalent). The main result of Section \ref{sec:q-T-G_invar}, Theorem~\ref{thm:q-T-G_Tutte_poly}, provides a definition of the Tutte polynomial for both matroids and $q$-matroids alike, analogous to Brylawski's definition. In Section~\ref{subsec:rank_poly}, we provide some results describing how the rank generating polynomial decomposes and is not a $q$-T-G invariant, which contrasts to the case for matroids.
Explicitly, in \cite[Section 6.1]{brylawski_oxley_1992} the notion of a T-G invariant is defined as follows. Recall that $M(e)$ is the restriction of $M$ to the interval $[\zero,e]$ for an atom $e\in\mathcal{L}$.

\begin{definition}\label{def:classicTG}
A \emph{T-G invariant} is a function $f$ defined on the class of matroids that satisfies the following three properties:
\makeatletter
\def\namedlabel#1#2{\begingroup
    #2%
    \def\@currentlabel{#2}%
    \phantomsection\label{#1}\endgroup
}
\makeatother
\begin{description}
    \item[\namedlabel{tg1}{{\rm (TG1)}}] For any pair of equivalent matroids, we have $f(M)=f(M')$.
    \item[\namedlabel{tg2}{{\rm (TG2)}}] If $e$ is a loop or an isthmus of $M$ then $f(M)=f(M(e))f(M-e)$. 
    \item[\namedlabel{tg3}{{\rm (TG3)}}] For any atom $e$ that is neither a loop nor an isthmus in $M$, we have $$f(M)=f(M-e)+f(M/e).$$ 
\end{description}
\end{definition}

In \cite[Theorem 6.2.2]{brylawski_oxley_1992} the Tutte polynomial is defined in terms of the T-G invariant properties as follows.
\begin{theorem}[Brylawski, 1972]
The Tutte polynomial is the unique function $\tau$ from the class of matroids to the ring $\mathbb{Z}[x,y]$ satisfying the following properties:
\begin{enumerate}
    \item $\tau(M(e))=x$ if $e$ is independent, and $\tau(M(e))=y$ if $e$ is a loop,
    \item $\tau$ is a T-G invariant.
\end{enumerate}
\end{theorem}

We will define a notion of such invariants in the more general setting of this paper and will show later that 
we recover the original T-G invariants for matroids by letting $q\rightarrow1$.

Recall that if $M$ is not prime-free, by Lemma \ref{lem:disjoint_Tutte_minors} there then exists an atom $e \in \mL$ such that every coloop of $M$ is contained in $[e,\one]$ and every loop of $M$ is contained in $[\zero,e^c]$, which extends to a partition of the form $\mathcal{P}(e)$.

\begin{notation}
  We let $\M$ denote either the class of all $q$-matroids or the class of all matroids.
\end{notation}

\begin{definition}\label{def:qTG}
Let $f$ be a function defined on $\M$. We say that $f$ is a \emph{$q$-Tutte-Grothendieck invariant} ($q$-T-G) if it satisfies the following three properties:
%%%%%%%%%%%%%%%%%%%%%%%%%%%%%%%%%
\makeatletter
\def\namedlabel#1#2{\begingroup
    #2%
    \def\@currentlabel{#2}%
    \phantomsection\label{#1}\endgroup
}
\makeatother
%%%%%%%%%%%%%%%%%%%%%%%%%%%%%%%%%%%
\begin{description}
    \item[\namedlabel{qP1}{{\rm ($q$-P1)}}] For all $M_1,M_2 \in \M$, if $M_1 \cong M_2$ then $f(M_1)=f(M_2)$,
    \item[\namedlabel{qP2}{{\rm ($q$-P2)}}] If $M$ is prime-free, then $f(M)=f(M(e))f(M/e)$ for any atom $e\in\mathcal{L}$,
    \item[\namedlabel{qP3}{{\rm ($q$-P3)}}] If $M$ is not prime-free, and $e$ is an atom of $M$ that is independent and contained in every coloop of $M$, then there exists a proper interval decomposition 
    $\mathcal{P}(e)$ such that
    \[f(M)=\sum_{I \in \mathcal{P}(e)}f(M(I))=
    f(M(e^c))+f(M/e)+\displaystyle\sum_{k=1}^{q^{n-1}-1}f(M([e_k,\bar{e}_k])).
    \]
\end{description}
\end{definition}

\begin{proposition}\label{prop:Tutte_poly_is_T-G_invar}
The Tutte polynomial $\tau(M;x,y)$ is a $q$-T-G invariant.
\end{proposition}

\begin{proof}
It is clear that $\tau(M;x,y)$ satisfies \ref{qP1}. If $M$ is prime-free, then $\tau(M;x,y)=x^{r(M)}y^{\nu(M)}$, and since any minor of $M$ is also prime-free, for any atom $e\in M$ we have that
\[
\tau(M/e;x,y)=x^{r(M)-r(e)}y^{\nu(M)-\nu(e)} \text{ and } \tau(M(e);x,y)=x^{r(e)}y^{\nu(e)}.
\] 
Therefore, $\tau(M;x,y)=\tau(M(e);x,y)\tau(M/e;x,y)$, which shows that $\tau(M;x,y)$ satisfies \ref{qP2}. If $M$ is not prime-free, then as in the proof of Theorem~\ref{thm:T-comp_partition_exist}, if $e\in\mL$ is an independent atom that is contained in all coloops of $M$, then 
there a Tutte partition $\mathcal{P}$ that is a refinement of a proper interval decomposition $\mathcal{P}(e)$. That is, $M$ has a Tutte partition of the form
\[
\mP = \mP'\;\dot{\cup}\; \mP'' \;\dot{\cup}\; \mP_1 \;\dot{\cup}\; \cdots\; \dot{\cup}\;\mP_{q^{n-1}-1}
\]
where $\mP'$ is a Tutte partition of the minor $M(e^c)$,  $\mP''$ is a Tutte partition of the minor $M/e$, and $\mP_k$ is a Tutte partition of the minor $M([e_k,\bar{e}_k])$ for each $k$.
Therefore, 
\begin{align*}
    \tau(M;x,y)&=\tau_\mP(M;x,y)\\
    &= \sum_{P \in \mP} x^{r_P(M(P))}y^{\nu_P(M(P))}\\
    &= \sum_{P \in \mP'} x^{r_P(M(P))}y^{\nu_P(M(P))} +\sum_{P \in \mP''} x^{r_P(M(P))}y^{\nu_P(M(P))} + \sum_{k=1}^{q^{n-1}-1} \sum_{P \in \mP_k} x^{r_P(M(P))}y^{\nu_P(M(P))}\\
    &=\tau(M(e^c);x,y)+\tau(M/e;x,y)+\displaystyle\sum_{k=1}^{q^{n-1}-1}\tau(M([e_k,\bar{e}_k]);x,y),
\end{align*}
which shows that $\tau(M;x,y)$ satisfies \ref{qP3}.
\end{proof}

We now arrive at a definition of the Tutte polynomial as a $q$-T-G invariant.

\begin{theorem}\label{thm:q-T-G_Tutte_poly}
The Tutte polynomial is the unique function $\tau$ from  $\mathcal{M}$ to the ring $\mathbb{Z}[x,y]$ satisfying the following properties:
\begin{description}
    \item[\namedlabel{Tpolyprop1}{\rm (1)}]$\tau(M(e);x,y)=x$ if $e\in\mL$ is independent in $M$, and $\tau(M(e);x,y)=y$ if $e\in\mL$ is a loop of $M$,
    \item[\namedlabel{Tpolyprop2}{\rm (2)}] $\tau(M;x,y)$ is a $q$-T-G invariant.
\end{description}
\end{theorem}
\begin{proof}
From Definition \ref{def:q-Tutte_poly} and Proposition~\ref{prop:Tutte_poly_is_T-G_invar}, it is clear that $\tau$ satisfies properties \ref{Tpolyprop1} and \ref{Tpolyprop2}, so all that remains to show is uniqueness. For $\h(\one)\leq2$ uniqueness clearly holds, so assume that it holds for any member of $\M$ whose support lattice has length at most $n-1$. Now let $\h(\one)=n$ and suppose that  $\gamma:\mathcal{M}\rightarrow\mathbb{Z}[x,y]$ is a function that satisfies properties \ref{Tpolyprop1} and \ref{Tpolyprop2}. 
Any proper minor of $M$ has support lattice of length at most $n-1$ and so by hypothesis the Tutte polynomial defined on these minors is the unique $q$-T-G invariant 
satisfying properties \ref{Tpolyprop1} and \ref{Tpolyprop2} of the theorem statement.
It follows that $\gamma(M(e);x,y)=\tau(M(e);x,y)$ and $\gamma(M/e;x,y)=\tau(M/e;x,y)$ for any atom $e \in \mL$, so in particular, if $M$ is prime-free we have
$$\gamma(M;x,y)=\gamma(M(e);x,y)\gamma(M/e;x,y)=\tau(M(e);x,y)\tau(M/e;x,y)=\tau(M;x,y),$$
for any atom $e\in M$. 
%%%%%%%%%
If $M$ is not prime-free, then by property \ref{Tpolyprop2} (and specifically, \ref{qP3}) we have 
\begin{align*}
    \gamma(M;x,y)&=\gamma(M(e^c);x,y)+\gamma(M/e;x,y)
    +\displaystyle\sum_{k=1}^{q^{n-1}-1}\gamma(M([e_k,\bar{e}_k];x,y)\\
    &=\tau(M(e^c);x,y)+\tau(M/e;x,y)+\displaystyle\sum_{k=1}^{q^{n-1}-1}\tau(M([e_k,\bar{e}_k]);x,y)\\
    &=\tau(M;x,y).\qedhere
\end{align*}
\end{proof}

\begin{remark}
A recursive definition of the characteristic polynomial of a $q$-matroid is shown in \cite[Theorem 5.14]{B22_proj} and in \cite[Corollary 3.4]{alfarano2023critical}, wherein the characteristic polynomial is expressed in terms of characteristic polynomials of minors of the $q$-matroid that are defined with respect to a partition of the one-dimensional spaces of ground vector space $V$. While this resembles the definition of a $q$-Tutte-Grothendieck invariant in this paper, the collection of intervals involved is different.
\end{remark}
%%%%%%%%%%%%%%%%%%%%%%%%%%%%%%%%%%%%%%%%%%%%%%%%%%%%%%%%%%%%%%%%%%%%%%%%%%%%%%
\subsection{$q$-Tutte-Grothendieck invariants and the rank generating polynomial}\label{subsec:rank_poly}

One of the fundamental properties of the rank generating polynomial for matroids is its decomposition as is shown in \cite{brylawski_oxley_1992} and \cite{C69}. In particular, the rank generating polynomial of a matroid is a T-G invariant in the sense of Definition \ref{def:classicTG}. We will see in Section \ref{subsec:T-G_invar} that for matroids, the $q$-T-G invariants and the T-G invariants coincide but for now we focus on $q$-T-G invariants. We will see a divergence between behaviour of the rank function of a matroid and that of a $q$-matroid: more precisely, we will see that in general, the rank function of a $q$-matroid is not a $q$-T-G invariant. 

For the remainder of this section, we assume that $\mL$ has a proper interval decomposition, that is, for any atom $e \in \mL$, we have a partition of 
$\mL$ into disjoint intervals:
\begin{equation}\label{eq:Se}
   \mathcal{P}(e) = [e,{\bf1}]\:\dot{\cup}\: [{\bf0},e^c]\:  \dot{\cup} \: [e_1,\Bar{e}_1] \:\dot{\cup} \:\cdots\:\dot{\cup}\:[e_{q^{n-1}-1},\bar{e}_{q^{n-1}-1}],
\end{equation}
such that $\h(e_k)=1$, $\h(\Bar{e}_k)=n-1$, and $e_k \leq \Bar{e}_k$.
Note that we make no assumptions about the atom $e$ here.
\begin{proposition}\label{thm:rank_poly_split}
Let $e\in\mathcal{L}$ be an atom and let 
$\mathcal{P}(e)$ be the proper interval decomposition of $\mL$, defined as in \rm{(\ref{eq:Se})}.
Then the rank generating polynomial $\rho(M;x,y)$ satisfies the following identity:
\begin{equation*}
    \rho(M;x,y)=x^{r({\bf1})-r(e^c)}\rho(M(e^c);x,y)+y^{\nu(e)}\rho(M/e;x,y)+\displaystyle\sum^{q^{n-1}-1}_{k=1}x^{r({\bf1})-r(\Bar{e}_k)}y^{\nu(e_k)}\rho(M([e_k,\Bar{e}_k]);x,y).
\end{equation*}
\end{proposition}

\begin{proof}
 For any interval $[a,b]\subseteq\mathcal{L}$, we have that $$\displaystyle\sum_{z\in[a,b]}x^{r({\bf1})-r(z)}y^{\nu(z)}=x^{r({\bf1})-r(b)}y^{\nu(a)}\displaystyle\sum_{z\in[a,b]}x^{r_{[a,b]}(b)-r_{[a,b]}(z)}y^{\nu_{[a,b]}(z)}=x^{r({\bf1})-r(b)}y^{\nu(a)}\rho(M([a,b]);x,y).$$ We then have
\begin{align*}
    \rho(M;x,y)&=\displaystyle\sum_{z\in\mathcal{L}}x^{r({\bf1})-r(z)}y^{\nu(z)}\\
    &=\displaystyle\sum_{z\in[e,{\bf1}]}x^{r({\bf1})-r(z)}y^{\nu(z)}+\displaystyle\sum_{z\in[{\bf0},e^c]}x^{r({\bf1})-r(z)}y^{\nu(z)}+\displaystyle\sum_{k=1}^{q^{n-1}-1}\sum_{z\in[e_k,\Bar{e}_k]}x^{r({\bf1})-r(z)}y^{\nu(z)}\\
    &=y^{\nu(e)}\rho(M/e;x,y)+x^{r({\bf1})-r(e^c)}\rho(M(e^c);x,y)\\
    &\hspace{0.5cm}+\displaystyle\sum_{k=1}^{q^{n-1}-1}x^{r({\bf1})-r(\Bar{e}_k)}y^{\nu(e_k)}\rho(M([e_k,\Bar{e}_k]);x,y).\qedhere
\end{align*}
\end{proof}

%%%%%%%%%%%%%%%%%%%%%%%%%%%%%%%%%%%%%%%%%%%%%%%%%%%%%%%%%%
\begin{example}
Let $M$ be the $q$-matroid 
of Example \ref{ex:P1tutte}. Let $\mathcal{P}(e)$ be the proper interval decomposition determined by $e=\langle101\rangle$, $e'=\langle100,010\rangle$, $e_1=\langle111\rangle$, $e_2=\langle011\rangle$, $e_3=\langle001\rangle$, $\Bar{e}_1=\langle100,011\rangle$, $\Bar{e}_2=\langle010,001\rangle$, and $\Bar{e}_3=\langle110,001\rangle$, which is shown on the right-hand side of 
Figure \ref{fig:tutte_example}.
%%%%%%%%%%%%
We express $\rho(M;x,y)$ as a sum of rank generating polynomials of the minors supported on the distinct elements of $\mathcal{P}(e)$.
\begin{align*}
    \rho(M;x,y)&=x+xy+7y+y^2+6\\
    &=(x+xy+y+2)+(y^2+3y+1)+3(y+1)\\
    &=\rho(M(e');x,y)+\rho(M/e;x,y)+\displaystyle\sum_{k=1}^3\rho(M([e_k,\Bar{e}_k]);x,y).
\end{align*}
\end{example}

\begin{remark}\label{lem:rank_tutte_intervals}
Let $\mathcal{P}$ be a collection of intervals that forms a partition of  $\mathcal{L}$.
Suppose that for each $A\in\mathcal{P}$ it holds that $r(\zero_A)=\h(\zero_A)$ and $r(\one_A)=r(\one)$. Note that we are not assuming that the elements of $\mP$ are prime-free and hence $\mP$ need not be a Tutte partition. Then we have that 
$$\rho(M;x,y)=\displaystyle\sum_{A\in\mathcal{P}}\rho(A;x,y).$$

To see this, let $A\subseteq\mL$ be any interval. We have $r_A(\one_A)-r_A(x)=r(\one_A)-r(x)$, and $\nu_A(x)=\h(x)-\h(\zero_A)-r_A(x)$ for all $x\in A$. 
Therefore, if $r(\zero_A)=\h(\zero_A)$ and $r(\one_A)=r({\bf1})$, then $r({\bf1})-r(x)=r_A(\one_A)-r_A(x)$ and $\nu(x)=\nu_A(x)$ for all $x\in A$. This now gives us
\begin{align*}
    \rho(M;x,y)&=\displaystyle\sum_{z\in\mathcal{L}}x^{r({\bf1})-r(z)}y^{\nu(z)}\\
    &=\displaystyle\sum_{A\in\mathcal{P}}\displaystyle\sum_{z\in A}x^{r({\bf1})-r(z)}y^{\nu(z)}\\
    &=\displaystyle\sum_{A\in\mathcal{P}}\displaystyle\sum_{z\in A}x^{r_A(\one_A)-r_A(z)}y^{\nu_A(z)}\\
    &=\displaystyle\sum_{A\in\mathcal{P}}\rho(A;x,y).
\end{align*}
\end{remark}

It is clear that the rank generating polynomial of a $q$-matroid satisfies \ref{qP1}. By Proposition~\ref{thm:rank_poly_split} we see that $\rho(M;x,y)$ also satisfies \ref{qP3}. 
However, if $M$ is a $q$-matroid, in general its rank generating polynomial does not satisfy \ref{qP2}. 

\begin{proposition}
      Let $\M$ denote the class of $q$-matroids. The rank function $\rho:\M \longrightarrow \Z[x,y]$ is not a $q$-T-G invariant.
\end{proposition}

\begin{proof}
    Suppose that $M$ is prime-free and let $e$ be an atom of $\mL$.
    Then $\rho(M/e;x,y)$ is also prime free and is supported on a lattice of height $n-1$. 
    We have $\rho(M;1,1)=\displaystyle\sum_{k=0}^n\qbin{n}{k}{q}$, since this counts the number of elements of $\mL$. However,
    \begin{align*}
    \rho(M(e);1,1)\rho(M/e;1,1)=2\cdot\displaystyle\sum_{k=0}^{n-1}\qbin{n-1}{k}{q}\neq\sum_{k=0}^n\qbin{n}{k}{q}=\rho(M;1,1).
\end{align*}
Therefore, $\rho(M;x,y)$ does not satisfy \ref{qP2} and hence the rank generating polynomial of a $q$-matroid  is not a $q$-T-G invariant.
\end{proof}

\section{A return to the matroid case}\label{sec:return_to_matroid_case}

In this section, we restrict to the matroid case. Firstly, in Section~\ref{subsec:Tutte-partition_Crapo}, we will look at the notion of a Tutte partition as defined in Section \ref{sec:q-Tutte_partition} and compare this with the partition arising in Crapo's definition of the Tutte polynomial \cite{C69}. The main result of Section~\ref{subsec:Tutte-partition_Crapo} is Corollary~\ref{cor:crapo-part_is_tutte-part}, which shows that Crapo's definition is a special case of Definition~\ref{def:Tutte_partition}.

Next, in Section~\ref{subsec:T-G_invar}, we will consider the axioms of a $q$-Tutte Grothendieck invariant (Definition \ref{def:qTG}), and compare these with the usual Tutte-Grothendieck axioms (Definition \ref{def:classicTG}). The main result of Section~\ref{subsec:T-G_invar} is Theorem~\ref{thm:T-G_P_equivalence}, which shows that Definition~\ref{def:qTG} gives a $q$-analogue of a T-G invariant.

\begin{notation}
  Throughout this section 
  $M$ denotes an arbitrary matroid with Boolean support lattice $\mL=\mL(S)$ for some finite set $S$ of cardinality $n$ and $\mathcal{B}$ denotes the collection of bases of $M$.  $\M$ now refers to the class of matroids.
\end{notation}

\subsection{Crapo's definition of the Tutte polynomial}\label{subsec:Tutte-partition_Crapo}

Recall that by Corollary \ref{cor:T-poly_indep_of_choice}, which holds both for matroids and $q$-matroids, the Tutte polynomial of Definition \ref{def:q-Tutte_poly} in the matroid case is necessarily the Tutte polynomial of \cite{C69}. 
The definition of the Tutte polynomial of a matroid given in \cite{C69} uses a partition $\mathcal{P}$ of $\mL$ into intervals such that each $A\in\mathcal{P}$ contains exactly one basis of $M$. This leads to the question of the equivalence
the partition of $\mL$ given in \cite{C69} and the Tutte partition of Definition~\ref{def:Tutte_partition}.

We now briefly describe the partition of $\mL$ into intervals arising in Crapo's paper. For each $b\in\mathcal{B}$, we define the interval $[b^-,b^+]$ in the following way.
Let $\beta$ denote any total ordering of the bases of $M$ derived from an arbitrary total ordering of the atoms as in \cite{C69}. For any interval $[x,y]\subseteq\mL$, let the $\beta$-max (respectively, $\beta$-min) basis of $[x,y]$ refer to the maximum (respectively, minimum) basis in $[x,y]$ according to this $\beta$. By \cite[Lemmas 7 and 8]{C69} we have that the basis $b$ is the $\beta$-max basis of $[b^-,{\bf1}]$ and $b$ is the $\beta$-min basis of $[{\bf0},b^+]$. We summarize these notions in the form of the following definition.

\begin{definition}\label{def:crapo_partition}
Let $\beta$ be an arbitrary total ordering of the bases of $M$. Let $\mathcal{P}$ be a partition of $M$ into a union of disjoint minors. We say that $\mP$ is a Crapo-Tutte partition if for every $[x,y]\in\mathcal{P}$ the following hold.
\begin{enumerate}
    \item There is exactly one basis $b$ of $M$ in $[x,y]$.
    \item The basis $b \in [x,y]$ is the $\beta$-max basis of $[x,{\bf1}]$.
    \item The basis $b \in [x,y]$ is the $\beta$-min basis of $[{\bf0},y]$.
\end{enumerate}
\end{definition}

We denote the elements of a Crapo-Tutte partition by $[b^-,b^+]$ for each $b \in \mB$ and note that each such interval is uniquely determined by the ordering $\beta$.
%We will use this characterisation in Example~\ref{ex:Tutte_vs_Crapo}. 
The definition of the Tutte polynomial given in \cite{C69} is as follows.
\begin{definition}
Let $M$ be a matroid with set of bases $\mathcal{B}$. The Tutte polynomial of $M$ is $$\tau(M;x,y)=\displaystyle\sum_{b\in\mathcal{B}}x^{|b-b^-|}y^{|b^+-b|}.$$
\end{definition}

\begin{remark}
   This definition does not immediately generalise to $q$-matroids: in the matroid case, by \cite[Lemma 8]{C69}, the collection of intervals $\{[b^-,b^+]\subseteq\mathcal{L}: b \in\mathcal{B}\}$ form a partition of $\mathcal{L}$, whereas such a construction does not yield a partition in the event that the support lattice is the subspace lattice.
\end{remark}

\begin{lemma}\label{lem:m-aisom/a}
Let $A \in \mL$. If each of the atoms in $A$ is either a loop or isthmus, then $M/A\cong M-A$.
\end{lemma}
\begin{proof}
Define the lattice isomorphism $\varphi:M-A\rightarrow M/A$ by $\varphi(x)=x\cup A$. %It is 
Clearly, if $a$ is an isthmus or loop then $r(a\cup B)=r(a)+r(B)$ for any $B \in \mL$ such that $a \notin B$. Therefore, for any $x\subseteq A^c$ we have that $r(x\cup A)=r(x)+r(A)$. We thus have that $r_{M/A}(\varphi(x))=r(\varphi(x))-r(A)=r(x)=r_{M-A}(x)$ for any $x\subseteq A^c$.
\end{proof}

\begin{lemma}\label{lem:one_base_prime_free}
For any $[x,y]\subseteq\mL$, $M([x,y])$ has only one basis if and only if $[x,y]$ is prime-free.
\end{lemma}
\begin{proof}
For one of the implications, it suffices to show that any matroid with only one basis is prime-free, so suppose that $M$ has only one basis. Suppose now that $M$ has a prime diamond $D$, say $D=[A,A\cup x\cup y]$ for some $A,x,y\in\mL$ with $\h(x)=\h(y)=1$ and $x \notin y \cup A$ and $y \notin x \cup A$. Since $M$ has only one basis, each atom in $\mL$ is either a loop or isthmus of $M$. Then from Lemma \ref{lem:m-aisom/a}, $M/A\cong M-A$. If there is a basis $B$ of $M/A$ such that $x\cup A$ and $y\cup A$ are both contained in $B$, then $A \cup x \cup y \subseteq B$, which contradicts the fact that $D$ is a prime diamond. Therefore, there are at least two bases in $M/A$ and so there are at least two bases in $M-A$. But any basis of $M-A$ corresponds to a unique basis of $M$, which yields a contradiction.

Conversely, suppose that $M$ contains two bases, $b_1,b_2\in\mathcal{B}$. By the basis-exchange property of matroids, there exist $b_3\in\mathcal{B}$ and atoms $e_1\in b_1-b_2$ and $e_2\in b_2-b_1$ such that $b_3=(b_1-e_1)\cup e_2$. This clearly implies that $[b_1-e_1,b_1\cup e_2]$ is a prime diamond.
\end{proof}

Recalling Definition~\ref{def:Tutte_partition}, we get the following corollary.

\begin{corollary}\label{cor:crapo-part_is_tutte-part}
Every Crapo-Tutte partition is a Tutte partition.
\end{corollary}
\begin{proof}
Let $\mP$ be a Crapo-Tutte partition of $\mL$ with respect to some ordering of $\mB$. Since $b^-\subseteq b$ and $b\subseteq b^+$ for any $b\in\mathcal{B}$, we have that $r(b^-)=\h(b^-)$ and $r(b^+)=r(M)$. Furthermore, by Lemma~\ref{lem:one_base_prime_free} we have that $[b^-,b^+]$ is prime-free. Therefore, $\mP=\{[b^-,b^+]: b \in \mB\}$ is a Tutte partition.
\end{proof}
%%%%%%%%%%%%%%%%%%%%%%%%%%%%%%%%%%%%%%%%%%%%%%%%%%%%%

%%%%%%%%%%%%%%%%%%%%%%%%%%%%%%%%%%%%%%%%%%%%%%%%%%%%%%%%%%%
An obvious question now arises, namely, whether or not every Tutte partition of a matroid is also a Crapo-Tutte partition. 
In fact, it is not the case that any Tutte partitions of a matroid can be represented as a disjoint union of intervals of the form $[b^-,b^+]$, $b\in\mathcal{B}$ for some total ordering of bases. 
The following example shows this.
\begin{example}\label{ex:Tutte_vs_Crapo}
Let $M=\mathcal{U}_{3,6}$ (the uniform matroid of rank three on six elements). For brevity, we write $123$ to denote the set $\{1,2,3\}$ etc. 
Consider the following Tutte partition $\mP$ of $M$, the elements of which are listed below.  
\begin{align*}
    &[456,\one]\\
    [124,12346], &[125,12345], [126,12356]\\
    [134,1345], [135,1356], [136,1346], &[234,2345], [235,2356], [236,2346]\\
    [24,245], [25,256], [26,246], &[34,346], [35,345], [36,356]\\
[4,145], &[5,156], [6,146]\\ 
&[\zero,123].
\end{align*}
(Note that for each $[a,b]\in\mathcal{P}$, we have that $a$ is independent, and $b$ is spanning, and that $\mathcal{P}$ is indeed a partition, thus satisfying the properties of a Tutte partition).

Suppose that $\mathcal{P}$ is a Crapo-Tutte partition and let $\leq_\beta$ denote its total order on the bases of $M$, as described in Definition~\ref{def:crapo_partition}. Consider the intervals $[4,145], [5,156]$, and $[6,146]$. By Property 2 of Definition~\ref{def:crapo_partition}, we have that $145$ is the $\beta\textup{-max basis of }[4,\one]$, $156$ is the $\beta\textup{-max basis of }[5,\one]$, and $146$ is the $\beta\textup{-max basis of }[6,\one]$. Since we also have that $146\in[4,\one]$, $145\in[5,\one]$, and $156\in[5,\one]$, we respectively have that $146\leq_\beta145$, $145\leq_\beta156$, and $156\leq_\beta146$. This clearly yields a contradiction, which means that $\mathcal{P}$ cannot be a Crapo-Tutte partition.
\end{example}

%%%%%%%%%%%%%%%%%%%%%%%%%%%%%%%%%%%%%%%%%%%%%%%%%%%%%%%%%%%%%%%%%%%%%%%%
\subsection{Tutte-Grothendieck invariants}\label{subsec:T-G_invar}

We now compare the axioms of a $q$-Tutte-Grothendieck invariant (Axioms \ref{qP1}-\ref{qP3}) and the classical axioms of a Tutte-Grothendieck invariant (Axioms \ref{tg1}-\ref{tg3}). We will show that for matroids, these sets of axioms are equivalent.

\begin{remark}\label{rem:double_contract_minor}
We point out that for any $d,e\in\mathcal{L}$, if $e\leq d$, then $(M/e)/d=M/d$.
To see this, first note that both matroids are defined on the interval $[d,{\bf1}]$ and then observe that for any $x\in[d,{\bf1}]$, $r_{(M/e)/d}(x)=r_{M/e}(x)-r_{M/e}(d)=r(x)-r(e)-(r(d)-r(e))=r(x)-r(d)=r_{M/d}(x).$
\end{remark}

%%%%%%%%%%%%%%%%%%%%%%%%%%%%%%%%%%%%%%%%%%%%%%%%%%%%%%%%%%%%%%%
\begin{theorem}\label{thm:T-G_P_equivalence}
Let $f:\M \longrightarrow \Z[x,y]$. Then $f$ is a $q$-T-G invariant if and only if $f$ is a T-G invariant.
\end{theorem}
\begin{proof}
Assume first that $f$ is a T-G invariant. It is immediate that \ref{qP1} is satisfied. If $M$ is prime-free, then by Lemma~\ref{lem:one_base_prime_free} it contains a unique basis, which means that each atom of $M$ is either a loop or an isthmus. Therefore, using Lemma~\ref{lem:m-aisom/a}, since $f$ satisfies property \ref{tg2}, it also satisfies property \ref{qP2}. If $M$ contains a prime diamond and $e$ is an independent atom that is contained in all coloops of $M$, then $e$ is neither a loop nor an isthmus and so since $f$ satisfies property \ref{tg3}, it necessarily satisfies property \ref{qP3} (since the summation $\sum_{k=1}^{q^{n-1}-1}f(M([e_k,\bar{e}_k]))$ is empty with $q=1$). It follows that any T-G invariant is also a $q$-T-G invariant.
 
Conversely, assume $f$ is a $q$-T-G invariant, so that $f$ satisfies properties \ref{qP1}, \ref{qP2}, and \ref{qP3}. Clearly, if $f$ satisfies \ref{qP1}, then it also satisfies \ref{tg1}. If $M$ contains an atom that is neither a loop nor an isthmus, then by Lemma~\ref{lem:one_base_prime_free}, $M$ contains a prime diamond, and thus \ref{tg3} is clearly satisfied. Therefore, we need only address the case for which $M$ is not prime-free, and $e\in M$ is a loop or an isthmus. Namely, we must show that in such a case $f(M)=f(M(e))f(M/e)$. 

If $\h(\one)=2$, then this holds trivially, so assume that this holds whenever $\h(\one)\leq n-1$. Now suppose that $M$ contains a prime diamond and $\h(\one)=n$. If $M$ contains a loop (respectively, an isthmus) $e$, then by \ref{qP2} there exists an independent atom $a$ such that $e\in M-a$ and $f(M)=f(M-a)+f(M/a)$. Clearly $e$ is a loop (respectively, an isthmus) of $M-a$, and $e\;\cup\; a$ is a loop (respectively, an isthmus) of $M/a$. By the induction hypothesis $f(M-a)=f((M-a)(e))f((M-a)-e)=f((M-a)(e))f(M-(e\;\cup\; a))$ and $f(M/a)=f((M/a)(e\;\cup\; a))f((M/e)-(e\;\cup\; a))$. Since $e$ is a loop (respectively, an isthmus) we have that $M(e)\cong (M-a)(e)\cong (M{/}a)(e\;\cup\; a)$, which means that $f(M(e))=f((M-a)(e))=f((M{/}a)(e\;\cup\; a))$ by ($q$-P1). It follows that
$$f(M)=f(M(e))(f((M-a)-e)+f((M/a)-(e\;\cup\;a))).$$
%%%%%%%%%%%%%%%%%
Since $a$ is not a loop (respectively, not an isthmus) in $M$ and all (respectively, none) of the bases are in $M/e$, $e\cup a$ is not a loop (respectively, an isthmus) of $M/e$, and $M/e$ is not prime-free. By \ref{qP3} we have that 
$f(M/e)=f((M/e)-(e\;\cup\; a))+f((M/e)/(e\cup a))$. 
Since $e$ is dependent (respectively, independent) to any lattice element in $M-a$, by Lemma \ref{lem:m-aisom/a} %and Remark \ref{rem:double_contract_minor} 
we have that $(M/e)-(e\cup a)\cong M-(e\cup a)$ and $(M/e)/(e\;\cup\; a)\cong (M/a)/(e\;\cup\; a)\cong (M/a)-(e\cup a)$. As a result, we have that
\begin{align*}
    f(M)&=f(M(e))(f(M-(e\cup a))+f((M/a)/(e\cup a)))\\
    &=f(M(e))(f((M/e)-(e\cup a))+f((M/e)/(e\cup a)))\\
    &=f(M(e))f(M/e).\qedhere
\end{align*}
\end{proof}

%%%%%%%%%%%%%%%%%%%%%%%%%%%%%%%%%%%%%%%%%%%%%%%%%%%%%%%%%%%%%%%%%%

\section{Further properties of the Tutte and rank generating polynomial}\label{sec:furtherpropertise}

In this section again $M$ may be either a matroid or a $q$-matroid. We present a number of properties of the Tutte polynomial of a $q$-matroid that are the $q$-analogues of already established properties of the Tutte polynomial for matroids.

\subsection{Duality properties}

Let $\varphi: \mL \longrightarrow \mL^*$ be a lattice anti-isomorphism. Recall that the dual of 
$M$ is defined to be $M^*=(\mL^*,r^*)$ where $r^*(\varphi(x)):=\h(\one)-\h(x)-r(M)+r(x)$. 
Moreover, the nullity function of $M^*$ is defined by $\nu^*(\varphi(x))=r(\one)-r(x)$ for each 
$x\in\mathcal{L}$.

\begin{proposition}\label{prop:dual}
\textcolor{white}{.}
{\bf }
   \begin{enumerate}
       \item $\rho(M^*;x,y)=\rho(M;y,x).$
       \item Let $M$ have a Tutte partition $\mP$. Then $\tau(M^*;x,y)=\tau(M;y,x).$
   \end{enumerate}
\end{proposition}

\begin{proof}
For the rank generating polynomial of $M^*$, we have that:
\[
\rho(M^*;x,y)=\displaystyle\sum_{z\in\mathcal{L}}x^{r^*(\varphi(\one))-r^*(\varphi(z))}y^{\nu^*(\varphi(z))}=\displaystyle\sum_{z\in\mathcal{L}}x^{\h(z)-r(z)}y^{r(M)-r(z)}=\rho(M;y,x).
\]

%Let $\mathcal{P}$ be a Tutte partition of $M$. 
We claim that $\mP^*$, the image of $\mP$ under $\varphi$, is a Tutte partition of $M^*$. 
For any interval $[a,b]$ supporting an element of $\mathcal{P}$, $r(a)=\h(a)$ and $r(b)=r(M)$. Therefore, 
\begin{align*}
   r^*(\varphi(a))&=\h(\one)-\h(a)-r(M)+r(a) = \h(\one)-r(\one)=\nu(\one),\\
   r^*(\varphi(b))&=\h(\one)-\h(b)-r(M)+r(b)=\h(\one)-\h(b)=\h(\varphi(b)).
\end{align*}
It remains to check is that $[\varphi(b),\varphi(a)]$ is prime-free. Since the image of any diamond in $\mathcal{L}$ is a diamond in $\mathcal{L}^*$, 
we need only consider the duals of matroidal diamonds. Clearly, the only diamond whose dual is prime is a prime itself and so $[\varphi(b),\varphi(a)]$ is prime-free. It follows that $\mP^*$ is a Tutte partition.
For each member $P \in \mP$, let $P^*$ denote its image under $\varphi$ in $P^*$. 
Then $r^*_{P^*}(P^*) = \nu(P)$ and $\nu_{P^*}^*(P^*)=r_P(P)$ and so we have:
\begin{align*}
    \tau_{\mP^*}(M^*;x,y) = \sum_{P^* \in \mP^*} x^{r^*_{P^*}(M^*(P^*))}y^{\nu_{P^*}^*(M^*(P^*))} = \sum_{P \in \mP} x^{\nu_P(M(P))}y^{r_P(M(P))}= \tau_{\mP}(M;y,x).
\end{align*}
The result now follows by the invariance of $\tau_{\mP}(M;x,y)$ under the choice of $\mP.$
\end{proof}

\begin{remark}
    In fact, Part 2 of Proposition \ref{prop:dual} follows easily from Part 1, since $\beta_q*\rho(M^*;x,y) =\beta_q*\rho(M;y,x)$.
    However, the proof above is independent of the rank generating polynomial.
\end{remark}

\subsection{Representations and evaluations}

We consider now structural properties of $M$ that can be extracted from its Tutte polynomial.

\begin{lemma}\label{lem:tutteeval11}
    Every Tutte partition of $M$ has the same cardinality.
\end{lemma}

\begin{proof}
Let $\mP$ be a Tutte partition of $M$. Then
   \begin{equation*}
       \tau(M;1,1)=\tau_{\mP}(M;1,1) = \sum_{P \in \mP} 1^{r_P(M(P))}1^{\nu_P(M(P))} = |\mP|.\qedhere
   \end{equation*}
\end{proof}

An immediate corollary of Lemma \ref{lem:tutteeval11} and Corollary \ref{cor:crapo-part_is_tutte-part}, is that in the case that $M$ is a matroid, $\tau(M;1,1)$ counts the number of bases of $M$. However, in the case that $M$ is a $q$-matroid the evaluation $\tau(M;1,1)$ does not return the number of bases. On the other hand, the evaluation of the rank generating function of $M$ at $(0,0)$ is given by $\rho(M;0,0) = |\{a \in \mL : r(a)=r(\one), \nu(a)=0\}|$, which is precisely the number of bases of $M$, just as in the case that $M$ is a matroid \cite[Proposition 6.2.1]{brylawski_oxley_1992}. Similarly, $\rho(M;1,0)$ gives the number of independent elements of $M$ in $\mL$.

Another way to represent the Tutte polynomial of $M$ is as a matrix.
If we gather together terms corresponding to prime-free minors of the same rank and nullity, we obtain:
\[
    \tau(M;x,y) = \sum_{i=0}^{r(\one)}\sum_{j=0}^{\nu(\one)} \tau_{i,j} x^{i}y^{j}. 
\]
Clearly, for each $(i,j)\in [r(\one)+1] \times [\nu(\one)+1]$, we have $\tau_{i,j}=|\{ P \in \mP : r_P(P)=i, \nu_P(P)=j\}|$ for any Tutte partition $\mP$: that is, the coefficients of $\tau(M;x,y)$ count the number of prime-free minors of a Tutte partition of a given rank and nullity.
Now consider the representation of these Tutte coefficients as an $(r(\one)+1) \times (\nu(\one)+1)$ matrix $\tau = ( \tau_{i,j} )$.

\begin{definition}\label{def:tracescalar}
    Let $\F$ be a field and let $s,t$ be positive integers. 
    We define a bilinear form $\langle \cdot, \cdot \rangle:\F^{s \times t}\times \F^{s \times t}\longrightarrow \F$ by
    $\langle A, B\rangle:=\tr(AB^T)$ for all $A,B \in \F^{s \times t}.$
\end{definition}

\noindent It is straightforward to check that for any $A,B \in \F^{s \times t},$ we have 
$\displaystyle \langle A, B\rangle:=\sum_{i=1}^s\sum_{j=1}^t A_{i,j}B_{i,j}$. 
We'll show that we can retrieve the number of bases of $M$ by an evaluation of the linear form $\tau:=\langle \tau,\cdot \rangle$.

\begin{proposition}
   Let $Q$ be the $(r(\one)+1) \times (\nu(\one)+1)$ matrix $Q = (q^{ij})$. Then the number of bases of $M$ is $\langle \tau, Q \rangle$.
\end{proposition}

\begin{proof}
    Let $\mP$ be a Tutte partition of $M$ and let $P \in \mP$. Then $P$ has a unique totally clopen element $c$, and every complement of $c$ in $P$ is a basis in $P$. Moreover, every such basis is also a basis in $M$, since $r(P)=r(\one)$. If $P$ has rank $i$ and nullity $j$, then the number of complements of $c$ in $P$ is $q^{ij}$. It follows that the total number of bases of $M$ is 
    \[
       \sum_{i=0}^{r(\one)} \sum_{j=0}^{\nu(\one)} \tau_{i,j} q^{ij} = \langle \tau, Q \rangle.\qedhere
     \]
\end{proof}
%%%%%%%%%%%%
Note that in the case that $M$ is a matroid, if we let $q\longrightarrow 1$ we retrieve that the number of bases of $M$ is
$\langle \tau, J \rangle = \tau(M;1,1)$, where $J$ denotes the all-ones matrix.
%%%%%%%%%%%%%%%%%%%%%%%%%%%%%%%%%%

Similarly, we have 
$\displaystyle \rho(M;x,y) = \sum_{a=0}^{r(\one)}\sum_{b=0}^{\nu(\one)} \rho_{a,b} x^a y^b,$ 
where $\rho_{a,b}$ counts the number of elements of $\mL$ of rank-lack $a$ and nullity $b$ in $M$, so we may represent the rank coefficents as the $(r(\one)+1) \times (\nu(\one)+1)$ matrix $\rho = (\rho_{a,b})$. If $E_{i,j} $ %=(\delta^{(i,j)}_{(a,b)})$ 
is the matrix with $(i,j)$-entry $1$ and zeroes elsewhere, then we have that the number of bases of $M$ is: 
\[
   \langle \tau, Q \rangle = \langle \beta_q * \rho, Q \rangle = \langle \alpha_q * \tau, E_{0,0} \rangle =
   \langle \rho, E_{0,0}\rangle.
\]
Proposition \ref{prop:dual} now says that the Tutte and rank functions of $M^*$ are obtained by matrix transposition $\tau^T$ and $\rho^T$ of the Tutte and rank functions of $M$, respectively.

\section{Towards a combinatorial characterisation of the Tutte polynomial for all $q$-matroids}
Theorem~\ref{thm:rank_to_tutte} gives a characterisation of a Tutte polynomial as a certain convolution of the rank generating polynomial of any $q$-matroid. However, Theorem~\ref{thm:T-comp_partition_exist} yields a combinatorial characterisation of the Tutte polynomial only for ($q$-)matroids whose support lattices admit a proper interval decomposition. In this section, we outline a limitation of Definition~\ref{def:Tutte_partition}, but we show a possible generalisation.

It is shown in \cite{HITZEMANN20103551} that there does not exist a proper interval decomposition of $\mathcal{L}(\mathbb{F}_2^n)$ for $n\geq 5$. Since our proof of the existence of a Tutte partition used proper interval decompositions, we cannot thus conclude that a Tutte partition always exists for $q$-matroids on such lattices. 
In Example~\ref{ex:no_tutte} we give a $q$-matroid that does not admit a Tutte partition. For this example, we need the following definition of the direct sum of $q$-matroids. Assume that all vector spaces in Definition~\ref{def:dir_sum} and Notation~\ref{not:dir_sum} are over the same field.

\begin{definition}[\cite{dir_sum_jur_cer}]\label{def:dir_sum}
    Let $M_1=(\mathcal{L}(V_1),r_1)$ and $M_2=(\mathcal{L}(V_2),r_2)$ be $q$-matroids such that $V_1$ has dimension $n_1$ and $V_2$ has dimension $n_2$. Let $V$ be an $(n_1+n_2)$-dimensional vector space and let $V_1$ and $V_2$ be identified with their embeddings in $V$ such that $V=V_1\oplus V_2$. The \emph{direct sum of $M_1$ and $M_2$}, denoted $M_1\oplus M_2$, is the $q$-matroid $(\mathcal{L}(V),r)$ such that
    $$r(A)=\min_{X\leq A}\{r_1(\pi_1(X))+r_2(\pi_2(X))+\text{dim}(A)-\text{dim}(X)\}$$
    where $\pi_i:V\rightarrow V_i$ is the canonical projection for $i=1,2$.
\end{definition}

%For Example~\ref{ex:no_tutte} w
We will also use the following convenient notation.

\begin{notation}\label{not:dir_sum}
    Let $V_1$, $V_2$, $\pi_1$, $\pi_2$ and $V$ be as in Definition~\ref{def:dir_sum}. Let $X_1\leq V_1$ and $X_2\leq V_2$. We write $X_1\oplus X_2$ to denote the subspace of $V$ such that $\pi_i(X_1\oplus X_2)=X_i$ for $i=1,2$. We use the notation $\mathcal{U}_{k,n}(q)$ to denote the uniform matroid $(\mathcal{L}(\mathbb{F}_q^n),r)$, which means that $r(A)=\dim(A)$ if $\dim(A)\leq k$ and has value $k$ otherwise. In particular, $r(\mathbf{1})=k$ in $\mathcal{U}_{k,n}(q)$.
\end{notation}

\begin{example}\label{ex:no_tutte}
    Let $\mathcal{U}_{1,2}(2)$ and $\mathcal{U}_{0,3}(2)$ be defined on the spaces $\langle e_1,e_2\rangle$ and $\langle e_3,e_4,e_5\rangle$ respectively. Let $M=(\mathcal{L},r):=\mathcal{U}_{1,2}(2)\oplus \mathcal{U}_{0,3}(2)$. Suppose that $M$ admits a Tutte partition $\mathcal{P}$. Since all loops must be contained in an interval in a Tutte partition, we must have that $[\mathbf{0},0\oplus\langle e_3,e_4,e_5\rangle]\subseteq I$ for some $I\in\mathcal{P}$ since $r(0\oplus\langle e_3,e_4,e_5\rangle)=0$. Since the maximal element of any interval in a Tutte partition must be spanning and since $r(\langle x\rangle\oplus\langle e_3,e_4,e_5\rangle)=1=r(\mathbf{1})$ for any $x\leq\langle e_1,e_2\rangle$, we may assume, without loss of generality, that $[\mathbf{0},\langle e_2\rangle\oplus\langle e_3,e_4,e_5\rangle]\subseteq I$. Since $M$ is not prime-free and $\langle e_2\rangle\oplus\langle e_3,e_4,e_5\rangle$ is a coatom, we deduce that $[\mathbf{0},\langle e_2\rangle\oplus\langle e_3,e_4,e_5\rangle]=I$. Since $r(\mathbf{1})=1$ and all of the loops of $M$ are contained in $I$, we have that the chain $C(a,\mathbf{1})$ is a clopen chain for any atom $a\in\mathcal{L}$ not in $I$. We can thus say, without loss of generality, that $J:=[\langle e_1\rangle\oplus0,\mathbf{1}]$ is in $\mathcal{P}$. By the same reasoning, the chain $C(a',c)$ is clopen for any atom $a'\in\mathcal{L}$ and coatom $c\in\mathcal{L}$ not in $I$ or $J$ (with $a'\leq c$). Therefore, all other intervals in $\mathcal{P}$ must have length three and be isomorphic to $U_{0,3}(2)$. We thus conclude that $\mathcal{P}$ is a proper interval decomposition of the subspace lattice of $\mathbb{F}_2^5$, which contradicts \cite[Theorem 3]{HITZEMANN20103551}. Therefore, $M$ does not admit a Tutte partition.
\end{example}

Example~\ref{ex:no_tutte} shows us that we indeed need a more general definition of the Tutte polynomial in order for it to apply to all $q$-matroids. We conjecture that the condition of the Tutte partition being a \emph{partition} is too strong. We give a proposed remedy for Example~\ref{ex:no_tutte} in the following example, which may lead to the desired definition of the Tutte polynomial.

\begin{example}\label{ex:tutte_remedy}
    Let $M$ be the $q$-matroid given in Example~\ref{ex:no_tutte}. Computing the rank generating polynomial of $M$ we get $\rho(M;x,y)=24+x+7xy+7xy^2+xy^3+148y+154y^2+31y^3+y^4$. After we apply the convolution $\beta$ to $\rho(M;x,y)$ we get $\beta*\rho(M;,x,y)=xy^3+y^4+15y^3$. It is possible, using Hall's theorem, to get a matching between the atoms and coatoms of $\mathcal{L}$ that are not contained in the intervals $I$ and $J$, which gives us a set of 15 intervals, all of which are isomorphic as minors of $M$ to $\mathcal{U}_{0,3}(2)$. By \cite[Theorem 3]{HITZEMANN20103551}, there must be some non-empty intersections between these intervals. Nonetheless, we clearly get a polynomial that coincides with $\beta*\rho(M;x,y)$ when we take the sum of the Tutte polynomials of these intervals with those of $I$ and $J$.
\end{example}

We have not in Example~\ref{ex:tutte_remedy} proposed a suggestion of a formal definition of the Tutte polynomial. However, this example suggests that it may be possible to successfully generalise the notion of a Tutte partition, which would give a combinatorial definition of the Tutte polynomial for arbitrary $q$-matroids.

\section{Acknowledgements}

The authors are very grateful to the anonymous referees for their insights and comments and for their careful reading of our manuscript. Their comments greatly improved this paper.
We are indebted to the authors of \cite{BCJ17}, namely Guus Bollen, Henry Crapo, and Relinde Jurrius, whose preprint provided an invaluable basis for this work. Our notion of a Tutte partition (see Definition \ref{def:Tutte_partition}) had its origins in \cite{BCJ17}, wherein it was proposed that the Tutte polynomial could be defined via a partition of the underlying lattice $\mL$ into prime-free minors.
  The inversion formula established in Theorem \ref{thm:rank_to_tutte} of this paper provides a way to establish that the Tutte polynomial is independent of the choice of Tutte partition via convolution with the rank generating polynomial. The formula itself was proposed as \cite[Conjecture 1]{BCJ17}.  
We thank Luke Postle for his helpful remarks and in particular for drawing our attention to \cite{HITZEMANN20103551}, which answered a question we had concerning the existence of proper interval decompositions in subspace lattices while revising this paper.
This publication has emanated from research
conducted with the financial support of Science Foundation Ireland (18/CRT/6049) and the European Union MSCA Doctoral Networks, (HORIZON-MSCA-2021-DN-01, Project 101072316).

\bibliographystyle{abbrv}
\bibliography{References}

\end{document}